\documentclass{article}
\usepackage{amssymb}
\usepackage{amsmath}
\usepackage{amsfonts}

\newtheorem{theorem}{Theorem}[section]

\newtheorem{condition}{Condition}[section]
\newtheorem{conjecture}{Conjecture}[section]
\newtheorem{corollary}{Corollary}[section]
\newtheorem{criterion}{Criterion}[section]
\newtheorem{definition}{Definition}[section]
\newtheorem{example}{Example}[section]

\newtheorem{proposition}{Proposition}[section]
\newtheorem{remark}{Remark}[section]

\newenvironment{proof}[1][Proof]{\noindent\textbf{#1.} }{\ \rule{0.5em}{0.5em}}
\renewcommand{\theequation}{\thesection.\arabic{equation}}
\input{tcilatex}
\begin{document}

\title{Categories and functors of universal algebraic geometry. Automorphic
equivalence of algebras.}
\author{A. Tsurkov, \\
Department of Mathematics, Federal University of Rio Grande\\
do Norte, Brazil, arkadytsurkov.2021@gmail.com}
\maketitle

\begin{abstract}
Universal algebraic geometry allows considering of geometric properties of
every universal algebra. When two algebras have same algebraic geometry? We
must consider the categories of algebraic closed sets of these algebras to
answer this question. The complete coincidence of these categories gives us
a concept of the geometric equivalence of algebras. Some type of
isomorphisms of these categories gives us a concept of the automorphic
equivalence of algebras. This concept has been considered since article of
B. Plotkin \cite{PlotkinSame} in 2003. We will give by language of category
theory one more elegant definition of this concept and recall some theorems
related to this concept.

\textit{Keywords: }Universal algebra, category theory, universal algebraic
geometry, automorphic equivalence of algebras.
\end{abstract}

\section{Introduction\label{intr}}

We can consider the universal algebraic geometry over the arbitrary variety
of algebras $\Theta $. We consider in this paper, for simplicity, only
varieties of one-sorted algebras.

We take an infinite countable set $X_{0}$. The free algebra $F\left(
X\right) $ with the set $X$ of free generators exists in our variety $\Theta 
$ for every subset $X\subset X_{0}$. We consider only free algebras $F\left(
X\right) $ with finite subset $X\subset X_{0}$. Theses algebras are objects
of the \textit{category }$\Theta ^{0}$\textit{\ of finitely generated free
algebras of the variety} $\Theta $, the homomorphisms of these algebras will
be morphisms of this category.

A pair $\left( f_{1},f_{2}\right) \in F\left( X\right) \times F\left(
X\right) $ for every $F\left( X\right) \in \mathrm{Ob}\Theta ^{0}$ can be
considered as an equation in our variety $\Theta $. The subset $T\subseteq
F\left( X\right) \times F\left( X\right) $ can be considered as a system of
equations. The elements of free algebra $F\left( X\right) $ are words of our
language, defined by signature of our variety and it's identities. So, when
we consider equations and their systems, we are dealing with the syntax of
our language.

We consider an arbitrary algebra $H\in \Theta $. The homomorphism\linebreak $%
\varphi \in \mathrm{Hom}\left( F\left( X\right) ,H\right) $, where $F\left(
X\right) \in \mathrm{Ob}\Theta ^{0}$ and $\left( f_{1},f_{2}\right) \in
F\left( X\right) \times F\left( X\right) $, can be consider as solution of
the equation $\left( f_{1},f_{2}\right) $ if $\varphi \left( f_{1}\right)
=\varphi \left( f_{2}\right) $ or, in other words, $\left(
f_{1},f_{2}\right) \in \ker \varphi $. The set $\left\{ \varphi \in \mathrm{%
Hom}\left( F\left( X\right) ,H\right) \mid T\subseteq \ker \varphi \right\} $
of all solutions of the system of equations $T\subseteq F\times F$ in the
algebra $H$ we denote by $T_{H}^{\prime }=\left\{ \varphi \in \mathrm{Hom}%
\left( F\left( X\right) ,H\right) \mid T\subseteq \ker \varphi \right\} $.
Solutions of systems of equations in some specific algebra $H$ are depend
from semantic of this algebra.

The set $\mathrm{Hom}\left( F\left( X\right) ,H\right) $ plays here the same
role that affine space $k^{n}$, where $k$ is a field, plays in classical
algebraic geometry. Every $\varphi \in \mathrm{Hom}\left( F\left( X\right)
,H\right) $ can be considered as a point of our affine space. We can
consider for every set of points $R\subseteq \mathrm{Hom}\left( F\left(
X\right) ,H\right) $ the equation $\left( f_{1},f_{2}\right) \in F\left(
X\right) \times F\left( X\right) $, which has as solutions in the algebra $H$
all points of $R$. It is clear, that in this case $\left( f_{1},f_{2}\right)
\in \bigcap\limits_{\varphi \in R}\ker \varphi $ holds. Therefore $%
\bigcap\limits_{\varphi \in R}\ker \varphi \subseteq F\left( X\right) \times
F\left( X\right) $, which we denote as $R_{H}^{\prime }$ is the maximal
system of equations, which has as solutions in the algebra $H$ all points of 
$R$. It is clear that $R_{H}^{\prime }$ is a congruence in $F\left( X\right) 
$.

Finally, for every system of equations $T\subseteq F\left( X\right) \times
F\left( X\right) $ we can consider an \textit{algebraic close} of this
system:%
\begin{equation}
\left( T_{H}^{\prime }\right) _{H}^{\prime }=T_{H}^{\prime \prime
}=\bigcap\limits_{\substack{ \varphi \in \mathrm{Hom}\left( F\left( X\right)
,H\right) ,  \\ T\subseteq \ker \varphi }}\ker \varphi \subseteq F\left(
X\right) \times F\left( X\right) .  \label{cl_def}
\end{equation}%
This is the maximal system of equations, which has in the algebra $H$ the
same solutions as system $T$.

\begin{definition}
We say that set (system of equations) $T\subseteq F\left( X\right) \times
F\left( X\right) $ is \textbf{algebraic closed}\textit{\ in }the algebra $H$
($H$\textbf{-closed}) if $T=T_{H}^{\prime \prime }$.
\end{definition}

Every algebraic closed set $T\subseteq F\left( X\right) \times F\left(
X\right) $ is a congruence in $F\left( X\right) $. The family of all $H$%
-closed sets in $F\left( X\right) \times F\left( X\right) $ we denote by $%
Cl_{H}\left( F\left( X\right) \right) $.

\section{Auxiliary facts from set theory and universal algebra. Some
notation.}

\setcounter{equation}{0}

In this paper we will work a lot with elements of Cartesian squares. So, we
will use the following short notations. If $A$ and $B$ are two sets and $%
\varphi :A\rightarrow B$ is a mapping, then we will denote $\varphi \left(
a_{1},a_{2}\right) =\left( \varphi \left( a_{1}\right) ,\varphi \left(
a_{2}\right) \right) \in B\times B$ for $\left( a_{1},a_{2}\right) \in
A\times A$. If $R\subseteq A\times A$, we denote $\varphi \left( R\right)
=\left\{ \varphi \left( r\right) \mid r\in R\right\} $. Also we denote for
every set $A$ the subset $\left\{ \left( a,a\right) \mid a\in A\right\}
\subset A\times A$ by $\Delta _{A}$.

We denote by $\left\vert A\right\vert $ the cardinality of the set $A$. We
denote in the same way the quantity of elements of a finite set.

\begin{remark}
\label{epi}If $\varphi _{1},\varphi _{2}:B\rightarrow C$ are two mappings
and $\tau :A\rightarrow B$ is a surjective mapping and $\varphi _{1}\tau
=\varphi _{2}\tau $ holds, then $\varphi _{1}=\varphi _{2}$.
\end{remark}

This is a well-known and trivial fact from set theory.

\begin{definition}
We say that a functor $\mathcal{F}:\mathfrak{A}\rightarrow \mathfrak{B}$
from a category $\mathfrak{A}$ to a category $\mathfrak{B}$ is a \textbf{%
surjective functor} if for every $B\in \mathrm{Ob}\mathfrak{B}$ there exists 
$A\in \mathrm{Ob}\mathfrak{A}$ such that $B=\mathcal{F}\left( A\right) $ and
for every $\beta \in \mathrm{Mor}_{\mathfrak{B}}\left( B_{1},B_{2}\right) $,
where $B_{1},B_{2}\in \mathrm{Ob}\mathfrak{B}$, there exists $\alpha \in 
\mathrm{Mor}_{\mathfrak{A}}\left( A_{1},A_{2}\right) $, such that $B_{1}=%
\mathcal{F}\left( A_{1}\right) $, $B_{2}=\mathcal{F}\left( A_{2}\right) $, $%
\beta =\mathcal{F}\left( \alpha \right) $.
\end{definition}

\begin{remark}
\label{epi_f}We have, as in Remark \ref{epi} that if $\mathcal{F}_{1},%
\mathcal{F}_{2}:\mathcal{B}\rightarrow \mathcal{C}$ are two functors and $%
\mathcal{T}:\mathcal{A}\rightarrow \mathcal{B}$ is a surjective functor and $%
\mathcal{F}_{1}\mathcal{T}=\mathcal{F}_{2}\mathcal{T}$ holds, then $\mathcal{%
F}_{1}=\mathcal{F}_{2}$.
\end{remark}

This remark is proved by considerations of corresponding mappings of objects
and functors.

Also in this paper we will work a lot with commutative diagrams of a
specific kind. We consider two algebras $A,B\in \Theta $ of the same variety 
$\Theta $, two congruences $T_{A}\subseteq A\times A$, $T_{B}\subseteq
B\times B$ and the diagram%
\begin{equation}
\begin{array}{ccc}
A & \overset{\tau _{A}}{\rightarrow } & A/T_{A} \\ 
\mu \downarrow &  & \downarrow \varphi \\ 
B & \underset{\tau _{B}}{\rightarrow } & B/T_{B}%
\end{array}%
,  \label{squear}
\end{equation}%
where $\tau _{A},\tau _{B}$\textit{\ }are natural epimorphisms. The set of
all homomorphisms $\mu $, which close the left column of this diagram
commutative for the fixed homomorphism $\varphi $, we will denote by $\left( 
\begin{array}{c}
\left( A,T_{A}\right) \\ 
\left( B,T_{B}\right)%
\end{array}%
,\varphi \right) $. The set of all homomorphisms $\varphi $, which close the
right column of this diagram commutative for the fixed homomorphism $\mu $,
we will denote $\left( \mu ,%
\begin{array}{c}
\left( A,T_{A}\right) \\ 
\left( B,T_{B}\right)%
\end{array}%
\right) $. If the set $\left( 
\begin{array}{c}
\left( A,T_{A}\right) \\ 
\left( B,T_{B}\right)%
\end{array}%
,\varphi \right) $ contains only one element $\mu $, we will denote $\mu
=\left( 
\begin{array}{c}
\left( A,T_{A}\right) \\ 
\left( B,T_{B}\right)%
\end{array}%
,\varphi \right) $. If the set $\left( \mu ,%
\begin{array}{c}
\left( A,T_{A}\right) \\ 
\left( B,T_{B}\right)%
\end{array}%
\right) $ contains only one element $\varphi $, we will denote $\varphi
=\left( \mu ,%
\begin{array}{c}
\left( A,T_{A}\right) \\ 
\left( B,T_{B}\right)%
\end{array}%
\right) $.

\begin{remark}
\label{nn_sym}If $\mu \in \left( 
\begin{array}{c}
\left( A,T_{A}\right) \\ 
\left( B,T_{B}\right)%
\end{array}%
,\varphi \right) $, then $\varphi \in \left( \mu ,%
\begin{array}{c}
\left( A,T_{A}\right) \\ 
\left( B,T_{B}\right)%
\end{array}%
\right) $ and, vice versa, if $\varphi \in \left( \mu ,%
\begin{array}{c}
\left( A,T_{A}\right) \\ 
\left( B,T_{B}\right)%
\end{array}%
\right) $, then $\mu \in \left( 
\begin{array}{c}
\left( A,T_{A}\right) \\ 
\left( B,T_{B}\right)%
\end{array}%
,\varphi \right) $.
\end{remark}

\begin{remark}
\label{nn_trans}We consider three algebras $A,B,C\in \Theta $ and three
congruences $T_{A}\subseteq A\times A$, $T_{B}\subseteq B\times B$, $%
T_{C}\subseteq C\times C$. If $\mu _{A}:A\rightarrow B$, $\mu
_{B}:B\rightarrow C$ are two homomorphisms and $\varphi _{A}\in \left( \mu
_{A},%
\begin{array}{c}
\left( A,T_{A}\right) \\ 
\left( B,T_{B}\right)%
\end{array}%
\right) $, $\varphi _{B}\in \left( \mu _{B},%
\begin{array}{c}
\left( B,T_{B}\right) \\ 
\left( C,T_{C}\right)%
\end{array}%
\right) $, then $\varphi _{B}\varphi _{A}\in \left( \mu _{B}\mu _{A},%
\begin{array}{c}
\left( A,T_{A}\right) \\ 
\left( C,T_{C}\right)%
\end{array}%
\right) $. If $\varphi _{A}:A/T_{A}\rightarrow B/T_{B}$, $\varphi
_{B}:B/T_{B}\rightarrow C/T_{C}$ are two homomorphisms and $\mu _{A}\in
\left( 
\begin{array}{c}
\left( A,T_{A}\right) \\ 
\left( B,T_{B}\right)%
\end{array}%
,\varphi _{A}\right) $, $\mu _{B}\in \left( 
\begin{array}{c}
\left( B,T_{B}\right) \\ 
\left( C,T_{C}\right)%
\end{array}%
,\varphi _{B}\right) $, then $\mu _{B}\mu _{A}\in \left( 
\begin{array}{c}
\left( A,T_{A}\right) \\ 
\left( C,T_{C}\right)%
\end{array}%
,\varphi _{B}\varphi _{A}\right) $.
\end{remark}

\begin{proposition}
\label{squear_prop}We consider again two algebras $A,B\in \Theta $ of the
same variety $\Theta $ and two congruences $T_{A}\subseteq A\times A$, $%
T_{B}\subseteq B\times B$.
\end{proposition}

\begin{enumerate}
\item[1)] \textit{If there exists a homomorphism }$\mu :A\rightarrow B$%
\textit{\ such that }$\mu \left( T_{A}\right) \subseteq T_{B}$\textit{\
holds, then there exists homomorphism }$\varphi :A/T_{A}\rightarrow B/T_{B}$%
\textit{\ such that the diagram (\ref{squear}) is commutative, i.e. }$%
\varphi \tau _{A}=\tau _{B}\mu $\textit{. In our notation we can write that
in this case by }$\left( \mu ,%
\begin{array}{c}
\left( A,T_{A}\right)  \\ 
\left( B,T_{B}\right) 
\end{array}%
\right) \neq \varnothing $.

\item[2)] \noindent \noindent \textit{Vice versa, if }$\varphi
:A/T_{A}\rightarrow B/T_{B}$ is a \textit{homomorphism and}\linebreak $\mu
\in \left( 
\begin{array}{c}
\left( A,T_{A}\right)  \\ 
\left( B,T_{B}\right) 
\end{array}%
,\varphi \right) $\textit{, then }$\mu \left( T_{A}\right) \subseteq T_{B}$%
\textit{.}

\item[3)] \textit{In the conditions of item 1)} $\left\vert \left( \mu ,%
\begin{array}{c}
\left( A,T_{A}\right)  \\ 
\left( B,T_{B}\right) 
\end{array}%
\right) \right\vert =1$.

\item[4)] $\left( id_{A},%
\begin{array}{c}
\left( A,T_{A}\right)  \\ 
\left( A,T_{A}\right) 
\end{array}%
\right) =id_{A/T_{A}}$\textit{.}
\end{enumerate}

\begin{proof}
1) The first item of this proposition is a well-known fact from universal
algebra. In module theory, it is known as the lemma of two short exact
sequences. Therefore, we will only give a proof of others items of this
proposition, which also don't cause any difficulties.

2) We consider the diagram (\ref{squear}). This diagram is commutative
because $\mu \in \left( 
\begin{array}{c}
\left( A,T_{A}\right)  \\ 
\left( B,T_{B}\right) 
\end{array}%
,\varphi \right) $. We suppose that $\left( a^{\prime },a^{\prime \prime
}\right) \in T_{A}$. The $\tau _{A}\left( a^{\prime }\right) =\tau
_{A}\left( a^{\prime \prime }\right) $ holds. Therefore, the%
\begin{equation*}
\tau _{B}\mu \left( a^{\prime }\right) =\varphi \tau _{A}\left( a^{\prime
}\right) =\varphi \tau _{A}\left( a^{\prime \prime }\right) =\tau _{B}\mu
\left( a^{\prime \prime }\right) 
\end{equation*}%
holds. So $\left( \mu \left( a^{\prime }\right) ,\mu \left( a^{\prime \prime
}\right) \right) \in T_{B}$.

3) If two homomorphisms $\varphi _{i}:A/T_{A}\rightarrow B/T_{B}$, $i=1,2$,
close the diagram (\ref{squear}) commutative then $\varphi _{1}\tau
_{A}=\tau _{B}\mu =\varphi _{2}\tau _{A}$ and $\varphi _{1}=\varphi _{2}$ by
Remark \ref{epi}.

4) We denote by $\tau $ the natural epimorphism $A\rightarrow A/T_{A}$. It
is clear that $id_{A}\left( T_{A}\right) =T_{A}$. So $\left\vert \left(
id_{A},%
\begin{array}{c}
\left( A,T_{A}\right)  \\ 
\left( A,T_{A}\right) 
\end{array}%
\right) \right\vert =1$ by 3). Also we have that $id_{A}\tau =\tau
=id_{A/T_{A}}\tau $. Therefore $id_{A/T_{A}}\in \left( id_{A},%
\begin{array}{c}
\left( A,T_{A}\right)  \\ 
\left( A,T_{A}\right) 
\end{array}%
\right) $.
\end{proof}

\begin{proposition}
\label{projective}If $i=1,2$, $F\left( X_{i}\right) \in \mathrm{Ob}\Theta
^{0}$, $T_{i}$ are congruences in $F\left( X_{i}\right) $ accordingly and $%
\varphi :F\left( X_{1}\right) /T_{1}\rightarrow F\left( X_{2}\right) /T_{2}$
is a homomorphism, then\linebreak $\left( 
\begin{array}{c}
\left( F\left( X_{1}\right) ,T_{1}\right)  \\ 
\left( F\left( X_{2}\right) ,T_{2}\right) 
\end{array}%
,\varphi \right) \neq \varnothing $.
\end{proposition}

This proposition is an elementary consequence of the well-known projective
property of free algebras.

\section{Categories and functors of universal algebraic geometry}

\setcounter{equation}{0}

We consider the variety of algebras $\Theta $. The category $\Theta ^{0}$ of
the finitely generated free algebras of this variety was defined in Section %
\ref{intr}.

It is also natural to consider for every algebra $H\in \Theta $ the
following two categories.

\subsection{The categories of $H$-closed congruences and of coordinate
algebras of the algebra $H$\label{categories}}

The first category $\mathfrak{Cl}_{H}$ is the category of the $H$-closed
congruences. The objects of this category are all couples $\left( F\left(
X\right) ,T\right) $, where $F\left( X\right) \in \mathrm{Ob}\Theta ^{0}$, $%
T\in Cl_{H}\left( F\left( X\right) \right) $. We define as follows the
morphisms of this category: if%
\begin{equation*}
\left( F\left( X_{1}\right) ,T_{1}\right) ,\left( F\left( X_{2}\right)
,T_{2}\right) \in \mathrm{Ob}\mathfrak{Cl}_{H}
\end{equation*}%
then%
\begin{equation*}
\mathrm{Mor}_{\mathfrak{Cl}_{H}}\left( \left( F\left( X_{1}\right)
,T_{1}\right) ,\left( F\left( X_{2}\right) ,T_{2}\right) \right) =\left\{
\varphi \in \mathrm{Hom}\left( F\left( X_{1}\right) ,F\left( X_{2}\right)
\right) \mid \varphi \left( T_{1}\right) \subseteq T_{2}\right\} .
\end{equation*}

\begin{remark}
\label{alg_syst}In other words, every object of this category is some
algebraic systems $F\left( X\right) \in \mathrm{Ob}\Theta ^{0}$ equipped by
algebraic operations of signature of our variety $\Theta $ and also by some
relation $T\in Cl_{H}\left( F\left( X\right) \right) $. The morphisms of
this category are homomorphisms of these algebraic systems (see \cite[I, 2.2]%
{Mal}) and the composition of these morphisms is a composition of the
corresponding homomorphisms.
\end{remark}

The second category $\mathfrak{Cor}_{H}$ is the category of the coordinate
algebras of the algebra $H$. The objects of this category are all quotient
algebras $F\left( X\right) /T$, where $F\left( X\right) \in \mathrm{Ob}%
\Theta ^{0}$, $T\in Cl_{H}\left( F\left( X\right) \right) $ and morphisms of
this category are all homomorphisms of these algebras. This is one example
of algebraic category.

\subsection{Factorization functor}

If $\left( F\left( X\right) ,T\right) \in \mathrm{Ob}\mathfrak{Cl}_{H}$,
then $F\left( X\right) /T\in \mathrm{Ob}\mathfrak{Cor}_{H}$. If%
\begin{equation*}
\left( F\left( X_{1}\right) ,T_{1}\right) ,\left( F\left( X_{2}\right)
,T_{2}\right) \in \mathrm{Ob}\mathfrak{Cl}_{H}
\end{equation*}%
and%
\begin{equation*}
\mu \in \mathrm{Mor}_{\mathfrak{Cl}_{H}}\left( \left( F\left( X_{1}\right)
,T_{1}\right) ,\left( F\left( X_{2}\right) ,T_{2}\right) \right)
\end{equation*}%
then $\mu \left( T_{1}\right) \subseteq T_{2}$. So, by Proposition \ref%
{squear_prop}, items 1) and 3), we have that $\left\vert \left( \mu ,%
\begin{array}{c}
\left( F\left( X_{1}\right) ,T_{1}\right) \\ 
\left( F\left( X_{2}\right) ,T_{2}\right)%
\end{array}%
\right) \right\vert =1$ and%
\begin{equation*}
\varphi =\left( \mu ,%
\begin{array}{c}
\left( F\left( X_{1}\right) ,T_{1}\right) \\ 
\left( F\left( X_{2}\right) ,T_{2}\right)%
\end{array}%
\right) \in \mathrm{Mor}_{\mathfrak{Cor}_{H}}\left( F\left( X_{1}\right)
/T_{1},F\left( X_{2}\right) /T_{2}\right) .
\end{equation*}

Therefore, it is a natural to define the mappings:%
\begin{equation*}
\left( \mathcal{FR}_{H}\right) _{\mathrm{Ob}}:\mathrm{Ob}\mathfrak{Cl}%
_{H}\rightarrow \mathrm{Ob}\mathfrak{Cor}_{H}
\end{equation*}%
by%
\begin{equation}
\left( \mathcal{FR}_{H}\right) _{\mathrm{Ob}}\left( F\left( X\right)
,T\right) =F\left( X\right) /T,  \label{FR_ob}
\end{equation}%
where $\left( F\left( X\right) ,T\right) \in \mathrm{Ob}\mathfrak{Cl}_{H}$,
and 
\begin{equation*}
\left( \mathcal{FR}_{H}\right) _{\mathrm{Mor}}:\mathrm{Mor}_{\mathfrak{Cl}%
_{H}}\left( \left( F\left( X_{1}\right) ,T_{1}\right) ,\left( F\left(
X_{2}\right) ,T_{2}\right) \right) \rightarrow
\end{equation*}%
\begin{equation*}
\mathrm{Mor}_{\mathfrak{Cor}_{H}}\left( \mathcal{FR}_{\mathrm{Ob}}\left(
F\left( X_{1}\right) ,T_{1}\right) ,\mathcal{FR}_{\mathrm{Ob}}\left( F\left(
X_{2}\right) ,T_{2}\right) \right)
\end{equation*}%
by%
\begin{equation}
\left( \mathcal{FR}_{H}\right) _{\mathrm{Mor}}\left( \mu \right) =\varphi
=\left( \mu ,%
\begin{array}{c}
\left( F\left( X_{1}\right) ,T_{1}\right) \\ 
\left( F\left( X_{2}\right) ,T_{2}\right)%
\end{array}%
\right) ,  \label{FR_mor_s}
\end{equation}%
where $\mu \in \mathrm{Mor}_{\mathfrak{Cl}_{H}}\left( \left( F\left(
X_{1}\right) ,T_{1}\right) ,\left( F\left( X_{2}\right) ,T_{2}\right)
\right) $.

\begin{proposition}
$\mathcal{FR}_{H}$ is a functor $\mathfrak{Cl}_{H}\rightarrow \mathfrak{Cor}%
_{H}$.
\end{proposition}

\begin{proof}
We consider $\left( F\left( X\right) ,T\right) \in \mathrm{Ob}\mathfrak{Cl}%
_{H}$. $id_{F\left( X\right) }:F\left( X\right) \rightarrow F\left( X\right) 
$ is an identity morphism of this object. The%
\begin{equation*}
\mathcal{FR}_{H}\left( id_{F\left( X\right) }\right) =\left( id_{F\left(
X\right) },%
\begin{array}{c}
\left( F\left( X\right) ,T\right) \\ 
\left( F\left( X_{2}\right) ,T\right)%
\end{array}%
\right) =id_{F\left( X\right) /T}
\end{equation*}%
holds by (\ref{FR_mor_s}) and by Proposition \ref{squear_prop} item 4).

We consider%
\begin{equation*}
\left( F\left( X_{1}\right) ,T_{1}\right) ,\left( F\left( X_{2}\right)
,T_{2}\right) ,\left( F\left( X_{3}\right) ,T_{3}\right) \in \mathrm{Ob}%
\mathfrak{Cl}_{H}
\end{equation*}%
and%
\begin{equation*}
\mu _{1}\in \mathrm{Mor}_{\mathfrak{Cl}_{H}}\left( \left( F\left(
X_{1}\right) ,T_{1}\right) ,\left( F\left( X_{2}\right) ,T_{2}\right)
\right) ,
\end{equation*}%
\begin{equation*}
\mu _{2}\in \mathrm{Mor}_{\mathfrak{Cl}_{H}}\left( \left( F\left(
X_{2}\right) ,T_{2}\right) ,\left( F\left( X_{3}\right) ,T_{3}\right)
\right) .
\end{equation*}%
\begin{equation*}
\mu _{2}\mu _{1}\in \mathrm{Mor}_{\mathfrak{Cl}_{H}}\left( \left( F\left(
X_{1}\right) ,T_{1}\right) ,\left( F\left( X_{3}\right) ,T_{3}\right)
\right) 
\end{equation*}%
holds.

\begin{equation*}
\left( \mathcal{FR}_{H}\right) \left( \mu _{i}\right) =\left( \mu _{i},%
\begin{array}{c}
\left( F\left( X_{i}\right) ,T_{i}\right)  \\ 
\left( F\left( X_{i+1}\right) ,T_{i+1}\right) 
\end{array}%
\right) ,
\end{equation*}%
where $i=1,2$, and%
\begin{equation*}
\left( \mathcal{FR}_{H}\right) \left( \mu _{2}\mu _{1}\right) =\left( \mu
_{2}\mu _{1},%
\begin{array}{c}
\left( F\left( X_{1}\right) ,T_{1}\right)  \\ 
\left( F\left( X_{3}\right) ,T_{3}\right) 
\end{array}%
\right) 
\end{equation*}%
hold by (\ref{FR_mor_s}). We have that%
\begin{equation*}
\left( \mathcal{FR}_{H}\right) \left( \mu _{2}\right) \left( \mathcal{FR}%
_{H}\right) \left( \mu _{1}\right) \in \left( \mu _{2}\mu _{1},%
\begin{array}{c}
\left( F\left( X_{1}\right) ,T_{1}\right)  \\ 
\left( F\left( X_{3}\right) ,T_{3}\right) 
\end{array}%
\right) 
\end{equation*}%
by Remark \ref{nn_trans}. Therefore%
\begin{equation*}
\left( \mathcal{FR}_{H}\right) \left( \mu _{2}\mu _{1}\right) =\left( 
\mathcal{FR}_{H}\right) \left( \mu _{2}\right) \left( \mathcal{FR}%
_{H}\right) \left( \mu _{1}\right) 
\end{equation*}%
by Proposition \ref{squear_prop}, item 3). This completes the proof.
\end{proof}

\begin{proposition}
\label{FR_epi}$\mathcal{FR}_{H}$ is an surjective functor.
\end{proposition}

\begin{proof}
If $F\left( X\right) /T\in \mathrm{Ob}\mathfrak{Cor}_{H}$, where $F\left(
X\right) \in \mathrm{Ob}\Theta ^{0}$, $T\in Cl_{H}F\left( X\right) $, then $%
\left( F\left( X\right) ,T\right) \in \mathrm{Ob}\mathfrak{Cl}_{H}$ and $%
\mathcal{FR}_{H}\left( F\left( X\right) ,T\right) =F\left( X\right) /T$ by (%
\ref{FR_ob}).

If $F\left( X_{i}\right) \in \mathrm{Ob}\Theta ^{0}$, $T_{i}\in
Cl_{H}F\left( X_{i}\right) $, $F\left( X_{i}\right) /T_{i}\in \mathrm{Ob}%
\mathfrak{Cor}_{H}$, $i=1,2$, and\linebreak $\varphi \in \mathrm{Mor}_{%
\mathfrak{Cor}_{H}}\left( F\left( X_{1}\right) /T_{1},F\left( X_{2}\right)
/T_{2}\right) $, where $F\left( X_{i}\right) /T_{i}=\mathcal{FR}_{H}\left(
F\left( X_{i}\right) ,T_{i}\right) $, $i=1,2$, then, by Proposition \ref%
{projective}, there exists a homomorphism $\mu :F\left( X_{1}\right)
\rightarrow F\left( X_{2}\right) $, such that $\mu \in \left( 
\begin{array}{c}
\left( F\left( X_{1}\right) ,T_{1}\right) \\ 
\left( F\left( X_{2}\right) ,T_{2}\right)%
\end{array}%
,\varphi \right) $. We have by Proposition \ref{squear_prop}, item 2), that $%
\mu \left( T_{1}\right) \subseteq T_{2}$. Hence $\mu \in \mathrm{Mor}_{%
\mathfrak{Cl}_{H}}\left( \left( F\left( X_{1}\right) ,T_{1}\right) ,\left(
F\left( X_{2}\right) ,T_{2}\right) \right) $. $\varphi \in \left( \mu ,%
\begin{array}{c}
\left( F\left( X_{1}\right) ,T_{1}\right) \\ 
\left( F\left( X_{2}\right) ,T_{2}\right)%
\end{array}%
\right) $ by Remark \ref{nn_sym}. $\left\vert \left( \mu ,%
\begin{array}{c}
\left( F\left( X_{1}\right) ,T_{1}\right) \\ 
\left( F\left( X_{2}\right) ,T_{2}\right)%
\end{array}%
\right) \right\vert =1$ by Proposition \ref{squear_prop}, item 3). So $%
\varphi =\left( \mu ,%
\begin{array}{c}
\left( F\left( X_{1}\right) ,T_{1}\right) \\ 
\left( F\left( X_{2}\right) ,T_{2}\right)%
\end{array}%
\right) =\left( \mathcal{FR}_{H}\right) _{\mathrm{Mor}}\left( \mu \right) $
by (\ref{FR_mor_s}).
\end{proof}

Also we can define the mappings:%
\begin{equation*}
\left( \mathcal{FG}_{H}\right) _{\mathrm{Ob}}:\mathrm{Ob}\mathfrak{Cl}%
_{H}\rightarrow \mathrm{Ob}\Theta ^{0}
\end{equation*}%
by 
\begin{equation}
\left( \mathcal{FG}_{H}\right) _{\mathrm{Ob}}\left( F\left( X\right)
,T\right) =F\left( X\right) \in \mathrm{Ob}\Theta ^{0}  \label{FG_ob}
\end{equation}%
for every $\left( F,T\right) \in \mathrm{Ob}\mathfrak{Cl}_{H}$, and 
\begin{equation*}
\left( \mathcal{FG}_{H}\right) _{\mathrm{Mor}}:\mathrm{Mor}_{\mathfrak{Cl}%
_{H}}\left( \left( F\left( X_{1}\right) ,T_{1}\right) ,\left( F\left(
X_{2}\right) ,T_{2}\right) \right) \rightarrow 
\end{equation*}%
\begin{equation*}
\mathrm{Mor}_{\Theta ^{0}}\left( \mathcal{FG}_{\mathrm{Ob}}\left( F\left(
X_{1}\right) ,T_{1}\right) ,\mathcal{FG}_{\mathrm{Ob}}\left( F\left(
X_{2}\right) ,T_{2}\right) \right) 
\end{equation*}%
by

\begin{equation}
\left( \mathcal{FG}_{H}\right) _{\mathrm{Mor}}\left( \varphi \right)
=\varphi \in \mathrm{Mor}_{\Theta ^{0}}\left( F\left( X_{1}\right) ,F\left(
X_{2}\right) \right) .  \label{FG_mor}
\end{equation}%
for every $\varphi \in \mathrm{Mor}_{\mathfrak{Cl}_{H}}\left( \left( F\left(
X_{1}\right) ,T_{1}\right) ,\left( F\left( X_{2}\right) ,T_{2}\right)
\right) $.

If $\varphi _{i}\in \mathrm{Mor}_{\mathfrak{Cl}_{H}}\left( \left( F\left(
X_{i}\right) ,T_{i}\right) ,\left( F\left( X_{i+1}\right) ,T_{i+1}\right)
\right) $ for $i=1,2$, then%
\begin{equation*}
\left( \mathcal{FG}_{H}\right) _{\mathrm{Mor}}\left( \varphi _{2}\varphi
_{1}\right) =\varphi _{2}\varphi _{1}=\left( \mathcal{FG}_{H}\right) _{%
\mathrm{Mor}}\left( \varphi _{2}\right) \left( \mathcal{FG}_{H}\right) _{%
\mathrm{Mor}}\left( \varphi _{1}\right) .
\end{equation*}%
The mapping $\left( \mathcal{FG}_{H}\right) _{\mathrm{Mor}}$ transforms the
identity morphism $id_{F\left( X\right) }:F\left( X\right) \rightarrow
F\left( X\right) $ of object $\left( F\left( X\right) ,T\right) \in \mathrm{%
Ob}\mathfrak{Cl}_{H}$ to the identity morphism $id_{F\left( X\right) }$ of
object $\left( \mathcal{FG}_{H}\right) _{\mathrm{Ob}}\left( F\left( X\right)
,T\right) =F\left( X\right) \in \mathrm{Ob}\Theta ^{0}$. So $\mathcal{FG}_{H}
$ is a functor.

\begin{proposition}
\label{allTheta_0_morpfism}If $\varphi \in \mathrm{Mor}_{\Theta ^{0}}\left(
F\left( X_{1}\right) ,F\left( X_{2}\right) \right) $ then%
\begin{equation*}
\varphi \in \mathrm{Mor}_{\mathfrak{Cl}_{H}}\left( \left( F\left(
X_{1}\right) ,T_{1}\right) ,\left( F\left( X_{2}\right) ,T_{2}\right)
\right) 
\end{equation*}%
for $T_{1}=\left( \Delta _{F\left( X_{1}\right) }\right) _{H}^{\prime \prime
}$ and every $T_{2}\in Cl_{H}\left( F\left( X_{2}\right) \right) $.
\end{proposition}

\begin{proof}
The $\Delta _{F\left( X_{1}\right) }\subseteq \ker \mu $ holds for every $%
\mu \in \mathrm{Hom}\left( F\left( X_{1}\right) ,H\right) $. Therefore%
\begin{equation}
\left( \Delta _{F\left( X_{1}\right) }\right) _{H}^{\prime \prime
}=\bigcap\limits_{\substack{ \mu \in \mathrm{Hom}\left( F\left( X_{1}\right)
,H\right)  \\ \Delta _{F\left( X_{1}\right) }\subseteq \ker \mu }}\ker \mu
=\bigcap\limits_{\mu \in \mathrm{Hom}\left( F\left( X_{1}\right) ,H\right)
}\ker \mu  \label{diag_close}
\end{equation}%
holds by (\ref{cl_def}).

We consider some $\left( f^{\prime },f^{\prime \prime }\right) \in \left(
\Delta _{F\left( X_{1}\right) }\right) _{H}^{\prime \prime }$. Here $%
f^{\prime },f^{\prime \prime }\in F\left( X_{1}\right) $. We have that $\nu
\varphi \left( f^{\prime }\right) =\nu \varphi \left( f^{\prime \prime
}\right) $ holds for every $\nu \in \mathrm{Hom}\left( F\left( X_{2}\right)
,H\right) $ by (\ref{diag_close}), because $\nu \varphi \in \mathrm{Hom}%
\left( F\left( X_{1}\right) ,H\right) $. Therefore,%
\begin{equation*}
\left( \varphi \left( f^{\prime }\right) ,\varphi \left( f^{\prime \prime
}\right) \right) \in \bigcap\limits_{\nu \in \mathrm{Hom}\left( F\left(
X_{2}\right) ,H\right) }\ker \nu =\left( \Delta _{F\left( X_{2}\right)
}\right) _{H}^{\prime \prime },
\end{equation*}%
which means that%
\begin{equation}
\varphi \left( \left( \Delta _{F\left( X_{1}\right) }\right) _{H}^{\prime
\prime }\right) \subseteq \left( \Delta _{F\left( X_{2}\right) }\right)
_{H}^{\prime \prime }  \label{id_incl}
\end{equation}%
holds. And, by (\ref{cl_def}), we have that%
\begin{equation}
\left( \Delta _{F\left( X_{2}\right) }\right) _{H}^{\prime \prime
}=\bigcap\limits_{\nu \in \mathrm{Hom}\left( F\left( X_{2}\right) ,H\right)
}\ker \nu \subseteq \bigcap\limits_{\substack{ \nu \in \mathrm{Hom}\left(
F\left( X_{2}\right) ,H\right)  \\ T_{2}\subseteq \ker \nu }}\ker \nu
=\left( T_{2}\right) _{H}^{\prime \prime }=T_{2}  \label{Diag_close_incl}
\end{equation}%
holds for every $T_{2}\in Cl_{H}\left( F\left( X_{2}\right) \right) $. Hence 
$\varphi \left( \left( \Delta _{F\left( X_{1}\right) }\right) _{H}^{\prime
\prime }\right) \subseteq T_{2}$ by (\ref{id_incl}) and (\ref%
{Diag_close_incl}). So $\varphi \in \mathrm{Mor}_{\mathfrak{Cl}_{H}}\left(
\left( F\left( X_{1}\right) ,T_{1}\right) ,\left( F\left( X_{2}\right)
,T_{2}\right) \right) $ for $T_{1}=\left( \Delta _{F\left( X_{1}\right)
}\right) _{H}^{\prime \prime }$ and every $T_{2}\in Cl_{H}\left( F\left(
X_{2}\right) \right) $.
\end{proof}

\begin{proposition}
$\mathcal{FG}_{H}$ is an surjective functor.
\end{proposition}

\begin{proof}
If $F\left( X\right) \in \mathrm{Ob}\Theta ^{0}$ then $\left( F\left(
X\right) ,\left( \Delta _{F\left( X\right) }\right) _{H}^{\prime \prime
}\right) \in \mathrm{Ob}\mathfrak{Cl}_{H}$ and $F\left( X\right) =\mathcal{FG%
}_{H}\left( F\left( X\right) ,\left( \Delta _{F\left( X\right) }\right)
_{H}^{\prime \prime }\right) $.

If $\varphi \in \mathrm{Mor}_{\Theta ^{0}}\left( F\left( X_{1}\right)
,F\left( X_{2}\right) \right) $ then%
\begin{equation*}
\varphi \in \mathrm{Mor}_{\mathfrak{Cl}_{H}}\left( \left( F\left(
X_{1}\right) ,\left( \Delta _{F\left( X_{1}\right) }\right) _{H}^{\prime
\prime }\right) ,\left( F\left( X_{2}\right) ,T_{2}\right) \right) 
\end{equation*}%
for every $T_{2}\in Cl_{H}F\left( X_{2}\right) $ by Proposition \ref%
{allTheta_0_morpfism}. Hence $\mathcal{FG}_{H}\left( \varphi \right)
=\varphi $, where on the left side $\varphi \in \mathrm{Mor}_{\mathfrak{Cl}%
_{H}}$ and on the right side $\varphi \in \mathrm{Mor}_{\Theta ^{0}}$.
\end{proof}

\section{When two algebras has same geometry?}

\setcounter{equation}{0}

We consider two algebras $H_{1},H_{2}\in \Theta $, and ask when they have
same geometry? It means, we ask when in these two algebras the transition
from syntax of our language to semantics of these algebras is carried out in
the same way. Or, more specifically and simply, when the transition from
system of equations to their solutions is carried out in the same way. The
answer depends on what we mean by the words "the same way".

The narrowest understanding of the coincidence of geometries of algebras is
the concept of geometric equivalence.

\begin{definition}
\cite{PlotkinVarCat}We say that \textit{algebras }$H_{1},H_{2}\in \Theta $%
\textit{\ are \textbf{geometrically equivalent}} if for every $F\left(
X\right) \in \mathrm{Ob}\Theta ^{0}$ the $Cl_{H_{1}}F\left( X\right)
=Cl_{H_{2}}F\left( X\right) \ $holds.
\end{definition}

\begin{proposition}
Two \textit{algebras }$H_{1},H_{2}\in \Theta $\textit{\ are }geometrically
equivalent if and only if $\mathfrak{Cl}_{H_{1}}=\mathfrak{Cl}_{H_{2}}$.
\end{proposition}

\begin{proof}
We suppose that for every $F\left( X\right) \in \mathrm{Ob}\Theta ^{0}$ the $%
Cl_{H_{1}}F\left( X\right) =Cl_{H_{2}}F\left( X\right) $\linebreak holds. If 
$\left( F\left( X\right) ,T\right) \in \mathrm{Ob}\mathfrak{Cl}_{H_{1}}$,
then $T\in Cl_{H_{1}}F\left( X\right) =Cl_{H_{2}}F\left( X\right) $. So $%
\left( F\left( X\right) ,T\right) \in \mathrm{Ob}\mathfrak{Cl}_{H_{2}}$ and $%
\mathrm{Ob}\mathfrak{Cl}_{H_{1}}\subseteq \mathrm{Ob}\mathfrak{Cl}_{H_{2}}$.
By symmetric reason $\mathrm{Ob}\mathfrak{Cl}_{H_{2}}\subseteq \mathrm{Ob}%
\mathfrak{Cl}_{H_{1}}$ and $\mathrm{Ob}\mathfrak{Cl}_{H_{1}}=\mathrm{Ob}%
\mathfrak{Cl}_{H_{2}}$.

If%
\begin{equation*}
\left( F\left( X_{1}\right) ,T_{1}\right) ,\left( F\left( X_{2}\right)
,T_{2}\right) \in \mathrm{Ob}\mathfrak{Cl}_{H_{1}}
\end{equation*}%
and%
\begin{equation*}
\varphi \in \mathrm{Mor}_{\mathfrak{Cl}_{H_{1}}}\left( \left( F\left(
X_{1}\right) ,T_{1}\right) ,\left( F\left( X_{2}\right) ,T_{2}\right)
\right) 
\end{equation*}%
then $F\left( X_{1}\right) ,F\left( X_{2}\right) \in \mathrm{Ob}\Theta ^{0}$%
, $\varphi \in \mathrm{Hom}\left( F\left( X_{1}\right) ,F\left( X_{2}\right)
\right) $, for $i=1,2$,%
\begin{equation*}
T_{i}\in Cl_{H_{1}}F\left( X_{i}\right) =Cl_{H_{2}}F\left( X_{i}\right) 
\end{equation*}%
holds and $\varphi \left( T_{1}\right) \subseteq T_{2}$, therefore%
\begin{equation*}
\varphi \in \mathrm{Mor}_{\mathfrak{Cl}_{H_{2}}}\left( \left( F\left(
X_{1}\right) ,T_{1}\right) ,\left( F\left( X_{2}\right) ,T_{2}\right)
\right) .
\end{equation*}%
So%
\begin{equation*}
\mathrm{Mor}_{\mathfrak{Cl}_{H_{1}}}\left( \left( F\left( X_{1}\right)
,T_{1}\right) ,\left( F\left( X_{2}\right) ,T_{2}\right) \right) \subseteq 
\mathrm{Mor}_{\mathfrak{Cl}_{H_{2}}}\left( \left( F\left( X_{1}\right)
,T_{1}\right) ,\left( F\left( X_{2}\right) ,T_{2}\right) \right) .
\end{equation*}%
By symmetric reason%
\begin{equation*}
\mathrm{Mor}_{\mathfrak{Cl}_{H_{2}}}\left( \left( F\left( X_{1}\right)
,T_{1}\right) ,\left( F\left( X_{2}\right) ,T_{2}\right) \right) \subseteq 
\mathrm{Mor}_{\mathfrak{Cl}_{H_{1}}}\left( \left( F\left( X_{1}\right)
,T_{1}\right) ,\left( F\left( X_{2}\right) ,T_{2}\right) \right) 
\end{equation*}%
and%
\begin{equation*}
\mathrm{Mor}_{\mathfrak{Cl}_{H_{1}}}\left( \left( F\left( X_{1}\right)
,T_{1}\right) ,\left( F\left( X_{2}\right) ,T_{2}\right) \right) =\mathrm{Mor%
}_{\mathfrak{Cl}_{H_{2}}}\left( \left( F\left( X_{1}\right) ,T_{1}\right)
,\left( F\left( X_{2}\right) ,T_{2}\right) \right) .
\end{equation*}

By Remark \ref{alg_syst}, the composition of morphisms both in the category $%
\mathfrak{Cl}_{H_{1}}$ and in the category $\mathfrak{Cl}_{H_{2}}$ is a
composition of the corresponding homomorphisms of algebraic systems equipped
by algebraic operations of signature of our variety $\Theta $ and also by
some binary relation. Therefore $\mathfrak{Cl}_{H_{1}}=\mathfrak{Cl}_{H_{2}}$%
.

Now we suppose that $\mathfrak{Cl}_{H_{1}}=\mathfrak{Cl}_{H_{2}}$. If $%
F\left( X\right) \in \mathrm{Ob}\Theta ^{0}$ and $T\in Cl_{H_{1}}F\left(
X\right) $. Then $\left( F\left( X\right) ,T\right) \in \mathrm{Ob}\mathfrak{%
Cl}_{H_{1}}=\mathrm{Ob}\mathfrak{Cl}_{H_{2}}$, so $T\in Cl_{H_{2}}F\left(
X\right) $ and $Cl_{H_{1}}F\left( X\right) \subseteq Cl_{H_{2}}F\left(
X\right) $. By symmetric reason $Cl_{H_{2}}F\left( X\right) \subseteq
Cl_{H_{1}}F\left( X\right) $, so $Cl_{H_{1}}F\left( X\right)
=Cl_{H_{2}}F\left( X\right) $.
\end{proof}

A notion of weak similarity is a wide generalization of the notion of
geometric equivalence algebras.

\begin{definition}
\textit{We say that} \textit{algebras }$H_{1},H_{2}\in \Theta $\textit{\ are 
\textbf{weak} \textbf{similar} if there is an isomorphism of categories: }$%
\Lambda :\mathfrak{Cl}_{H_{1}}\rightarrow \mathfrak{Cl}_{H_{2}}$.
\end{definition}

This concept is very wide, so studying its properties seems to as to be a
challenging task. The first reasonable step is, in our opinion, the study of
properties of automorphic equivalence of algebras. This concept is a subject
of this paper.

\section{Automorphic equivalence of algebras}

\setcounter{equation}{0}

We suppose from here onwards that our variety $\Theta $ fulfills the

\begin{condition}
\label{monoiso}$\Phi \left( F\left( x_{1}\right) \right) \cong F\left(
x_{1}\right) $ for every automorphism $\Phi $ of the category $\Theta ^{0}$.
\end{condition}

Many really interesting varieties of algebras satisfy this condition. For
example, this condition is fulfilled by the variety of all semigroups, by
all subvarieties of the variety of all groups, by all subvarieties of the
variety of all linear algebras and so on. For us an linear algebra over a
field $k$ is a vector space a field $k$ which is also is a ring, which is
not necessarily commutative, associative, not necessarily with unit, with a
bilinear multiplication.

The notion of automorphic equivalence is wider than the notion of geometric
equivalence, but narrower than the notion of similarity of algebras.

The notion of automorphic equivalence of algebras was discussed from \cite%
{PlotkinSame}. All definitions of this notion were somewhat unnatural and
cumbersome. For example, in \cite[Definition 4.3]{TsurkovManySorted} we can
see this

\begin{definition}
\label{old}We say that \textit{an automorphism }$\Phi :\Theta
^{0}\rightarrow \Theta ^{0}$ provides the \textit{\textbf{automorphic
equivalence }}of\textit{\textbf{\ }algebras }$H_{1},H_{2}\in \Theta $\textit{%
\ (or algebras }$H_{1},H_{2}$ are\textit{\ \textbf{automorphically
equivalent }}by \textit{an automorphism }$\Phi $)\textit{\ if there exist
bijections}%
\begin{equation*}
\alpha (\Phi )_{F\left( X\right) }:Cl_{H_{1}}(F\left( X\right) )\rightarrow
Cl_{H_{2}}(\Phi (F\left( X\right) ))
\end{equation*}%
for every $F\left( X\right) \in \mathrm{Ob}\Theta ^{0}$, \textit{coordinated
in the following sense: if }$F\left( X_{1}\right) ,F\left( X_{2}\right) \in 
\mathrm{Ob}\Theta ^{0}$\textit{, }$\mu _{1},\mu _{2}\in \mathrm{Hom}\left(
F\left( X_{1}\right) ,F\left( X_{2}\right) \right) $\textit{, }$T\in
Cl_{H_{1}}(F\left( X_{2}\right) )$\textit{\ then}%
\begin{equation}
\tau \mu _{1}=\tau \mu _{2},  \label{old_1}
\end{equation}%
\textit{if and only if }%
\begin{equation}
\widetilde{\tau }\Phi \left( \mu _{1}\right) =\widetilde{\tau }\Phi \left(
\mu _{2}\right) ,  \label{old_2}
\end{equation}%
\textit{where }$\tau :F\left( X_{2}\right) \rightarrow F\left( X_{2}\right)
/T$\textit{, }$\widetilde{\tau }:\Phi \left( F\left( X_{2}\right) \right)
\rightarrow \Phi \left( F\left( X_{2}\right) \right) /\alpha (\Phi
)_{F\left( X_{2}\right) }\left( T\right) $\textit{\ are the natural
epimorphisms.}
\end{definition}

This Definition is some clarification of the Definition from \cite%
{PlotkinSame}. But it also looks unnatural and bulky.

\begin{remark}
\label{phi_inv_eq}If algebras $H_{1},H_{2}\in \Theta $ are subjects of
Definition \ref{old} with automorphism\textit{\ }$\Phi :\Theta
^{0}\rightarrow \Theta ^{0}$ then algebras $H_{2},H_{1}$ are subjects of
Definition \ref{old} with automorphism\textit{\ }$\Phi ^{-1}$. It was proved
in \cite[Corollary 1 from Theorem 4.1]{TsurkovManySorted}.
\end{remark}

In this paper we will give an other definition of automorphic equivalence
which seems clearer and more natural to us. We use for this purpose the
language of category theory.

\begin{definition}
\label{new}We say that two algebras $H_{1},H_{2}\in \Theta $ are \textbf{%
automorphically equivalent} if there are three isomorphisms: $\Lambda :%
\mathfrak{Cl}_{H_{1}}\rightarrow \mathfrak{Cl}_{H_{2}}$, $\Phi :\Theta
^{0}\rightarrow \Theta ^{0}$, $\Psi :\mathfrak{Cor}_{H_{1}}\rightarrow 
\mathfrak{Cor}_{H_{2}}$such that the diagrams%
\begin{equation*}
\begin{array}{ccc}
\mathfrak{Cl}_{H_{1}} & \overrightarrow{\Lambda } & \mathfrak{Cl}_{H_{2}} \\ 
\downarrow \mathcal{FG}_{H_{1}} &  & \mathcal{FG}_{H_{2}}\downarrow \\ 
\Theta ^{0} & \underrightarrow{\Phi } & \Theta ^{0}%
\end{array}%
,%
\begin{array}{ccc}
\mathfrak{Cl}_{H_{1}} & \overrightarrow{\Lambda } & \mathfrak{Cl}_{H_{2}} \\ 
\downarrow \mathcal{FR}_{H_{1}} &  & \mathcal{FR}_{H_{2}}\downarrow \\ 
\mathfrak{Cor}_{H_{1}} & \underrightarrow{\Psi } & \mathfrak{Cor}_{H_{2}}%
\end{array}%
\end{equation*}%
are commutative.
\end{definition}

Our goal is to prove that Definition \ref{old} and Definition \ref{new} are
equivalent.

If $\Phi $ is an automorphism of the category $\Theta ^{0}$, then there
exists a system of bijections%
\begin{equation*}
\left\{ s_{F\left( X\right) }^{\Phi }:F\left( X\right) \rightarrow \Phi
\left( F\left( X\right) \right) \mid F\left( X\right) \in \mathrm{Ob}\Theta
^{0}\right\} 
\end{equation*}%
such that%
\begin{equation}
\Phi \left( \mu \right) =s_{F\left( X_{2}\right) }^{\Phi }\mu \left(
s_{F\left( X_{1}\right) }^{\Phi }\right) ^{-1}  \label{autom_action}
\end{equation}%
holds for every $\mu \in \mathrm{Mor}_{\Theta ^{0}}\left( F\left(
X_{1}\right) ,F\left( X_{2}\right) \right) $ (see \cite{PlZhit}, \cite[%
Theorem 2.1]{TsurkovManySorted}). (\ref{autom_action}) is equivalent to the
commutativity of the diagram%
\begin{equation*}
\begin{array}{ccc}
F\left( X_{1}\right)  & \overrightarrow{s_{F\left( X_{1}\right) }^{\Phi }} & 
\Phi \left( F\left( X_{1}\right) \right)  \\ 
\downarrow \mu  &  & \Phi \left( \mu \right) \downarrow  \\ 
F\left( X_{2}\right)  & \underrightarrow{s_{F\left( X_{2}\right) }^{\Phi }}
& \Phi \left( F\left( X_{2}\right) \right) 
\end{array}%
.
\end{equation*}

If $\left\{ s_{F\left( X\right) }^{\Phi }:F\left( X\right) \rightarrow \Phi
\left( F\left( X\right) \right) \mid F\left( X\right) \in \mathrm{Ob}\Theta
^{0}\right\} $ is a system of bijections which fulfills (\ref{autom_action})
for automorphism $\Phi $ of the category $\Theta ^{0}$, then the system of
bijections $\left\{ s_{F\left( X\right) }^{\Phi ^{-1}}:F\left( X\right)
\rightarrow \Phi ^{-1}\left( F\left( X\right) \right) \mid F\left( X\right)
\in \mathrm{Ob}\Theta ^{0}\right\} $, where%
\begin{equation}
s_{F\left( X\right) }^{\Phi ^{-1}}=\left( s_{\Phi ^{-1}\left( F\left(
X\right) \right) }^{\Phi }\right) ^{-1}  \label{phi_inv_bij}
\end{equation}%
fulfills (\ref{autom_action}) for automorphism $\Phi ^{-1}$. It was remarked
in \cite[proof of Corollary 1 from Theorem 4.1]{TsurkovManySorted}.

\begin{proposition}
\label{prop_phi_phi_t_incl}If $F\left( X_{1}\right) ,F\left( X_{2}\right)
\in \mathrm{Ob}\Theta ^{0}$, $\mu \in \mathrm{Hom}\left( F\left(
X_{1}\right) ,F\left( X_{2}\right) \right) $, $T_{i}\subseteq F\left(
X_{i}\right) \times F\left( X_{i}\right) $, $i=1,2$, $\mu \left(
T_{1}\right) \subseteq T_{2}$, $\Phi $ is an automorphism of the category $%
\Theta ^{0}$ which fulfills (\ref{autom_action}), then $\Phi \left( \mu
\right) \left( s_{F\left( X_{1}\right) }^{\Phi }\left( T_{1}\right) \right)
\subseteq s_{F\left( X_{2}\right) }^{\Phi }\left( T_{2}\right) $.
\end{proposition}

\begin{proof}
\begin{equation*}
\Phi (\mu )\left( s_{F\left( X_{1}\right) }^{\Phi }\left( T_{1}\right)
\right) =s_{F\left( X_{2}\right) }^{\Phi }\mu \left( s_{F\left( X_{1}\right)
}^{\Phi }\right) ^{-1}s_{F\left( X_{1}\right) }^{\Phi }\left( T_{1}\right) =
\end{equation*}%
\begin{equation*}
s_{F\left( X_{2}\right) }^{\Phi }\mu \left( T_{1}\right) \subseteq
s_{F\left( X_{2}\right) }^{\Phi }\left( T_{2}\right) 
\end{equation*}%
by (\ref{autom_action}).
\end{proof}

\subsection{Consequences of the Definition \protect\ref{old}}

In this subsection we suppose that two algebras $H_{1},H_{2}\in \Theta $ are
subjects of Definition \ref{old}.

\begin{equation}
\alpha (\Phi )_{F\left( X\right) }\left( T\right) =s_{F\left( X\right)
}^{\Phi }\left( T\right) ,  \label{alpha}
\end{equation}%
holds by \cite[Theorem 4.1]{TsurkovManySorted} for every $T\in
Cl_{H_{1}}(F\left( X\right) )$, where%
\begin{equation*}
\left\{ s_{F\left( X\right) }^{\Phi }:F\left( X\right) \rightarrow \Phi
\left( F\left( X\right) \right) \mid F\left( X\right) \in \mathrm{Ob}\Theta
^{0}\right\}
\end{equation*}%
is a system of bijections which fulfills (\ref{autom_action}) for the
automorphism $\Phi $.

\begin{remark}
\label{alpha_unic}A system of bijections $\left\{ s_{F\left( X\right)
}^{\Phi }\mid F\left( X\right) \in \mathrm{Ob}\Theta ^{0}\right\} $ is not
uniquely defined by automorphism $\Phi $, but by \cite[Corollary 1 from
Proposition 4.1]{TsurkovManySorted} the bijections%
\begin{equation*}
\left\{ \alpha (\Phi )_{F\left( X\right) }:Cl_{H_{1}}(F\left( X\right)
)\rightarrow Cl_{H_{2}}(\Phi (F\left( X\right) ))\mid F\left( X\right) \in 
\mathrm{Ob}\Theta ^{0}\right\}
\end{equation*}%
are uniquely defined by automorphism $\Phi $.
\end{remark}

\subsubsection{The functor $\Lambda $\label{functor_lambda}}

If $\left( F\left( X\right) ,T\right) \in \mathrm{Ob}\mathfrak{Cl}_{H_{1}}$,
then $F\left( X\right) \in \mathrm{Ob}\Theta ^{0}$, $T\in Cl_{H_{1}}(F\left(
X\right) )$,%
\begin{equation*}
\alpha (\Phi )_{F\left( X\right) }\left( T\right) =s_{F\left( X\right)
}^{\Phi }\left( T\right) \in Cl_{H_{2}}(\Phi (F\left( X\right) )),
\end{equation*}%
by (\ref{alpha}). So $\left( \Phi (F\left( X\right) ),s_{F\left( X\right)
}^{\Phi }\left( T\right) \right) \in \mathrm{Ob}\mathfrak{Cl}_{H_{2}}$.

We define%
\begin{equation*}
\Lambda _{\mathrm{Ob}}:\mathrm{Ob}\mathfrak{Cl}_{H_{1}}\rightarrow \mathrm{Ob%
}\mathfrak{Cl}_{H_{2}}
\end{equation*}%
by%
\begin{equation}
\Lambda _{\mathrm{Ob}}\left( F\left( X\right) ,T\right) =\left( \Phi
(F\left( X\right) ),s_{F\left( X\right) }^{\Phi }\left( T\right) \right) .
\label{delta_ob}
\end{equation}

Now we will define%
\begin{equation*}
\Lambda _{\mathrm{Mor}}:\mathrm{Mor}_{\mathfrak{Cl}_{H_{1}}}\left( \left(
F\left( X_{1}\right) ,T_{1}\right) ,\left( F\left( X_{2}\right)
,T_{2}\right) \right) \rightarrow 
\end{equation*}%
\begin{equation*}
\mathrm{Mor}_{\mathfrak{Cl}_{H_{2}}}\left( \Lambda _{\mathrm{Ob}}\left(
F\left( X_{1}\right) ,T_{1}\right) ,\Lambda _{\mathrm{Ob}}\left( F\left(
X_{2}\right) ,T_{2}\right) \right) 
\end{equation*}%
for every $\left( F\left( X_{1}\right) ,T_{1}\right) ,\left( F\left(
X_{2}\right) ,T_{2}\right) \in \mathrm{Ob}\mathfrak{Cl}_{H_{1}}$. If%
\begin{equation*}
\varphi \in \mathrm{Mor}_{\mathfrak{Cl}_{H_{1}}}\left( \left( F\left(
X_{1}\right) ,T_{1}\right) ,\left( F\left( X_{2}\right) ,T_{2}\right)
\right) 
\end{equation*}%
then $F\left( X_{1}\right) ,F\left( X_{2}\right) \in \mathrm{Ob}\Theta ^{0}$%
, $T_{1}\in Cl_{H_{1}}F\left( X_{1}\right) $, $T_{2}\in Cl_{H_{1}}F\left(
X_{2}\right) $, $\varphi \left( T_{1}\right) \subseteq T_{2}$ holds by
definition of $\mathfrak{Cl}_{H_{1}}$. $\Lambda _{\mathrm{Ob}}\left( F\left(
X_{i}\right) ,T_{i}\right) =\left( \Phi (F\left( X_{i}\right) ),s_{F\left(
X_{i}\right) }^{\Phi }\left( T_{i}\right) \right) $, $i=1,2$ by (\ref%
{delta_ob}). $s_{F\left( X_{i}\right) }^{\Phi }\left( T_{i}\right) \in
Cl_{H_{2}}\left( \Phi (F\left( X_{i}\right) )\right) $, $i=1,2$ as above.
Therefore, by Proposition \ref{prop_phi_phi_t_incl}, $\Phi (\varphi )\left(
s_{F\left( X_{1}\right) }^{\Phi }\left( T_{1}\right) \right) \subseteq
s_{F\left( X_{2}\right) }^{\Phi }\left( T_{2}\right) $. Hence,%
\begin{equation*}
\Phi \left( \varphi \right) \in \mathrm{Mor}_{\mathfrak{Cl}_{H_{2}}}\left(
\Lambda _{\mathrm{Ob}}\left( F\left( X_{1}\right) ,T_{1}\right) ,\Lambda _{%
\mathrm{Ob}}\left( F\left( X_{2}\right) ,T_{2}\right) \right) ,
\end{equation*}%
and we define%
\begin{equation}
\Lambda _{\mathrm{Mor}}\left( \varphi \right) =\Phi \left( \varphi \right) .
\label{lambda_mor}
\end{equation}

\begin{proposition}
\label{lambda_functor}$\Lambda :\mathfrak{Cl}_{H_{1}}\rightarrow \mathfrak{Cl%
}_{H_{2}}$ is a functor.
\end{proposition}

\begin{proof}
If $\left( F\left( X\right) ,T\right) \in \mathrm{Ob}\mathfrak{Cl}_{H_{1}}$
and $id_{F\left( X\right) }:F\left( X\right) \rightarrow F\left( X\right) $
identity morphism of this object, then $\Lambda _{\mathrm{Mor}}\left(
id_{F\left( X\right) }\right) =\Phi \left( id_{F\left( X\right) }\right)
=id_{\Phi \left( F\left( X\right) \right) }$ is an identity morphism of $%
\Lambda _{\mathrm{Ob}}\left( F\left( X\right) ,T\right) =\left( \Phi
(F\left( X\right) ),s_{F\left( X\right) }^{\Phi }\left( T\right) \right) $
by (\ref{delta_ob}) and (\ref{lambda_mor}).

Now we consider%
\begin{equation*}
\left( F\left( X_{1}\right) ,T_{1}\right) ,\left( F\left( X_{2}\right)
,T_{2}\right) ,\left( F\left( X_{3}\right) ,T_{3}\right) \in \mathrm{Ob}%
\mathfrak{Cl}_{H_{1}}
\end{equation*}%
and%
\begin{equation*}
\varphi _{1}\in \mathrm{Mor}_{\mathfrak{Cl}_{H_{1}}}\left( \left( F\left(
X_{1}\right) ,T_{1}\right) ,\left( F\left( X_{2}\right) ,T_{2}\right)
\right) ,
\end{equation*}%
\begin{equation*}
\varphi _{2}\in \mathrm{Mor}_{\mathfrak{Cl}_{H_{1}}}\left( \left( F\left(
X_{2}\right) ,T_{2}\right) ,\left( F\left( X_{3}\right) ,T_{3}\right)
\right) .
\end{equation*}%
By (\ref{lambda_mor}) we have that%
\begin{equation*}
\Lambda \left( \varphi _{1}\right) \Lambda \left( \varphi _{2}\right) =\Phi
\left( \varphi _{1}\right) \Phi \left( \varphi _{2}\right) =\Phi \left(
\varphi _{1}\varphi _{2}\right) =\Lambda \left( \varphi _{1}\varphi
_{2}\right) .
\end{equation*}%
It means that $\Lambda $ is a functor.
\end{proof}

\begin{proposition}
\label{left_comm}The functor $\Lambda $ closes the left diagram in
Definition \ref{new} commutative.
\end{proposition}

\begin{proof}
We have by (\ref{delta_ob}) and (\ref{FG_ob}) for every $\left( F\left(
X\right) ,T\right) \in \mathrm{Ob}\mathfrak{Cl}_{H_{1}}$, that%
\begin{equation*}
\mathcal{FG}_{H_{2}}\left( \Lambda \left( F\left( X\right) ,T\right) \right)
=\mathcal{FG}_{H_{2}}\left( \Phi (F\left( X\right) ),s_{F\left( X\right)
}^{\Phi }\left( T\right) \right) =\Phi (F\left( X\right) )
\end{equation*}%
and%
\begin{equation*}
\Phi \left( \mathcal{FG}_{H_{1}}\left( F\left( X\right) ,T\right) \right)
=\Phi (F\left( X\right) ).
\end{equation*}%
Also we have by (\ref{lambda_mor}) and (\ref{FG_mor}) for every%
\begin{equation*}
\varphi \in \mathrm{Mor}_{\mathfrak{Cl}_{H_{1}}}\left( \left( F\left(
X_{1}\right) ,T_{1}\right) ,\left( F\left( X_{2}\right) ,T_{2}\right)
\right) ,
\end{equation*}%
where $\left( F\left( X_{1}\right) ,T_{1}\right) ,\left( F\left(
X_{2}\right) ,T_{2}\right) \in \mathrm{Ob}\mathfrak{Cl}_{H_{1}}$, that%
\begin{equation*}
\mathcal{FG}_{H_{2}}\left( \Lambda \left( \varphi \right) \right) =\mathcal{%
FG}_{H_{2}}\left( \Phi \left( \varphi \right) \right) =\Phi \left( \varphi
\right) 
\end{equation*}%
and%
\begin{equation*}
\Phi \left( \mathcal{FG}_{H_{1}}\left( \varphi \right) \right) =\Phi \left(
\varphi \right) .
\end{equation*}
\end{proof}

\begin{proposition}
\label{lambda_isom}$\Lambda $ is an isomorphism of categories.
\end{proposition}

\begin{proof}
If the automorphism $\Phi :\Theta ^{0}\rightarrow \Theta ^{0}$ provides the
automorphic equivalence of algebras $H_{1}$ and $H_{2}$, then the
automorphism $\Phi ^{-1}:\Theta ^{0}\rightarrow \Theta ^{0}$ provides the
automorphic equivalence of algebras $H_{2}$ and $H_{1}$ by Remark \ref%
{phi_inv_eq}. So, there exists a functor $\Lambda ^{\prime }$ which is also
defined by (\ref{delta_ob}) and (\ref{lambda_mor}) with respect to the
automorphism $\Phi ^{-1}$. We have, by (\ref{delta_ob}) and (\ref%
{phi_inv_bij}), that 
\begin{equation*}
\Lambda ^{\prime }\Lambda \left( F\left( X\right) ,T\right) =\Lambda
^{\prime }\left( \Phi (F\left( X\right) ),s_{F\left( X\right) }^{\Phi
}\left( T\right) \right) =
\end{equation*}%
\begin{equation*}
\left( \Phi ^{-1}\Phi (F\left( X\right) ),s_{\Phi (F\left( X\right) )}^{\Phi
^{-1}}\left( s_{F\left( X\right) }^{\Phi }\left( T\right) \right) \right) =
\end{equation*}%
\begin{equation*}
\left( F\left( X\right) ,\left( s_{\Phi ^{-1}\left( \Phi (F\left( X\right)
)\right) }^{\Phi }\right) ^{-1}\left( s_{F\left( X\right) }^{\Phi }\left(
T\right) \right) \right) =\left( F\left( X\right) ,T\right) 
\end{equation*}%
for every $\left( F\left( X\right) ,T\right) \in \mathrm{Ob}\mathfrak{Cl}%
_{H_{1}}$. Also we have for every%
\begin{equation*}
\varphi \in \mathrm{Mor}_{\mathfrak{Cl}_{H_{1}}}\left( \left( F\left(
X_{1}\right) ,T_{1}\right) ,\left( F\left( X_{2}\right) ,T_{2}\right)
\right) ,
\end{equation*}%
where $\left( F\left( X_{1}\right) ,T_{1}\right) ,\left( F\left(
X_{2}\right) ,T_{2}\right) \in \mathrm{Ob}\mathfrak{Cl}_{H_{1}}$, that%
\begin{equation*}
\Lambda ^{\prime }\Lambda \left( \varphi \right) =\Phi ^{-1}\Phi \left(
\varphi \right) =\varphi 
\end{equation*}%
by (\ref{lambda_mor}). So%
\begin{equation}
\Lambda ^{\prime }\Lambda =id_{\mathfrak{Cl}_{H_{1}}}.  \label{lambda_inv_1}
\end{equation}%
Symmetrical reasoning gives us that%
\begin{equation}
\Lambda \Lambda ^{\prime }=id_{\mathfrak{Cl}_{H_{2}}}.  \label{lambda_inv_2}
\end{equation}%
Therefore $\Lambda $ is an isomorphism of categories.
\end{proof}

\subsubsection{The functor $\Psi $}

Every object of the category $\mathfrak{Cor}_{H_{1}}$ has form $F\left(
X\right) /T$, where $F\left( X\right) \in \mathrm{Ob}\Theta ^{0}$, $T\in
Cl_{H_{1}}F\left( X\right) $. Hence, as remarked in the beginning of
Subsection \ref{functor_lambda},%
\begin{equation*}
\alpha (\Phi )_{F\left( X\right) }\left( T\right) =s_{F\left( X\right)
}\left( T\right) \in Cl_{H_{2}}\left( \Phi \left( F\left( X\right) \right)
\right) .
\end{equation*}%
So $\Phi \left( F\left( X\right) \right) /s_{F\left( X\right) }^{\Phi
}\left( T\right) \in \mathrm{Ob}\mathfrak{Cor}_{H_{2}}$. We define%
\begin{equation*}
\Psi _{\mathrm{Ob}}:\mathrm{Ob}\mathfrak{Cor}_{H_{1}}\rightarrow \mathrm{Ob}%
\mathfrak{Cor}_{H_{2}}
\end{equation*}%
by%
\begin{equation}
\Psi _{\mathrm{Ob}}\left( F\left( X\right) /T\right) =\Phi \left( F\left(
X\right) \right) /s_{F\left( X\right) }^{\Phi }\left( T\right) .
\label{psi_ob}
\end{equation}%
Now we will define%
\begin{equation*}
\Psi _{\mathrm{Mor}}:\mathrm{Mor}_{\mathfrak{Cor}_{H_{1}}}\left( F\left(
X_{1}\right) /T_{1},F\left( X_{2}\right) /T_{2}\right) \rightarrow 
\end{equation*}%
\begin{equation*}
\mathrm{Mor}_{\mathfrak{Cor}_{H_{2}}}\left( \Psi _{\mathrm{Ob}}\left(
F\left( X_{1}\right) /T_{1}\right) ,\Psi _{\mathrm{Ob}}\left( F\left(
X_{2}\right) /T_{2}\right) \right) 
\end{equation*}%
for every $F\left( X_{1}\right) /T_{1},F\left( X_{2}\right) /T_{2}\in 
\mathrm{Ob}\mathfrak{Cor}_{H_{1}}$.

To this end, we must prove the following

\begin{proposition}
\label{psi_mor_prop}We consider $F\left( X_{1}\right) /T_{1},F\left(
X_{2}\right) /T_{2}\in \mathrm{Ob}\mathfrak{Cor}_{H_{1}}$ and a homomorphism 
$\varphi :F\left( X_{1}\right) /T_{1}\rightarrow F\left( X_{2}\right) /T_{2}$%
. There is only one homomorphism $\nu =\left( \Phi \left( \mu \right) ,%
\begin{array}{c}
\left( \Phi \left( F\left( X_{1}\right) \right) ,s_{F\left( X_{1}\right)
}^{\Phi }\left( T_{1}\right) \right) \\ 
\left( \Phi \left( F\left( X_{2}\right) \right) ,s_{F\left( X_{2}\right)
}^{\Phi }\left( T_{2}\right) \right)%
\end{array}%
\right) $ for every $\mu \in \left( 
\begin{array}{c}
\left( F\left( X_{1}\right) ,T_{1}\right) \\ 
\left( F\left( X_{2}\right) ,T_{2}\right)%
\end{array}%
,\varphi \right) $.
\end{proposition}

\begin{proof}
If $\mu \in \left( 
\begin{array}{c}
\left( F\left( X_{1}\right) ,T_{1}\right) \\ 
\left( F\left( X_{2}\right) ,T_{2}\right)%
\end{array}%
,\varphi \right) $, then $\mu \left( T_{1}\right) \subseteq T_{2}$ holds by
Proposition \ref{squear_prop}, item 2). So $\Phi \left( \mu \right) \left(
s_{F\left( X_{1}\right) }^{\Phi }\left( T_{1}\right) \right) \subseteq
s_{F\left( X_{2}\right) }^{\Phi }\left( T_{2}\right) $ holds by Proposition %
\ref{prop_phi_phi_t_incl}. So%
\begin{equation}
\left( \Phi \left( \mu \right) ,%
\begin{array}{c}
\left( \Phi \left( F\left( X_{1}\right) \right) ,s_{F\left( X_{1}\right)
}^{\Phi }\left( T_{1}\right) \right) \\ 
\left( \Phi \left( F\left( X_{2}\right) \right) ,s_{F\left( X_{2}\right)
}^{\Phi }\left( T_{2}\right) \right)%
\end{array}%
\right) =\nu _{\mu }  \label{nu_mu}
\end{equation}%
by Proposition \ref{squear_prop}, item 1) and 3).

We denote by $\tau _{i}:F\left( X_{i}\right) \rightarrow F\left(
X_{i}\right) /T_{i}$ and\linebreak $\widetilde{\tau }_{i}:\Phi \left(
F\left( X_{i}\right) \right) \rightarrow \Phi \left( F\left( X_{i}\right)
\right) /s_{F\left( X_{i}\right) }^{\Phi }\left( T_{i}\right) $ the natural
epimorphisms, where $i=1,2$. If $\mu _{i}\in \left( 
\begin{array}{c}
\left( F\left( X_{1}\right) ,T_{1}\right)  \\ 
\left( F\left( X_{2}\right) ,T_{2}\right) 
\end{array}%
,\varphi \right) $, where $i=1,2$, then $\tau _{2}\mu _{i}=\varphi \tau _{1}$%
. So $\tau _{2}\mu _{1}$ $=\tau _{2}\mu _{2}$. Therefore $\widetilde{\tau }%
_{2}\Phi \left( \mu _{1}\right) =\widetilde{\tau }_{2}\Phi \left( \mu
_{2}\right) $ holds by Definition \ref{old}. $\nu _{\mu _{i}}\widetilde{\tau 
}_{1}=\widetilde{\tau }_{2}\Phi \left( \mu _{i}\right) $, where $i=1,2$,
hold by \ref{nu_mu}. Hence $\nu _{\mu _{1}}\widetilde{\tau }_{1}=\nu _{\mu
_{2}}\widetilde{\tau }_{1}$, so $\nu _{\mu _{1}}=\nu _{\mu _{2}}$ by Remark %
\ref{epi}.
\end{proof}

$\Phi \left( F\left( X_{i}\right) \right) /s_{F\left( X_{i}\right) }^{\Phi
}\left( T_{i}\right) =\Psi _{\mathrm{Ob}}\left( F\left( X_{i}\right)
/T_{i}\right) $, $i=1,2$, by (\ref{psi_ob}). So%
\begin{equation*}
\nu =\left( \Phi \left( \mu \right) ,%
\begin{array}{c}
\left( \Phi \left( F\left( X_{1}\right) \right) ,s_{F\left( X_{1}\right)
}^{\Phi }\left( T_{1}\right) \right)  \\ 
\left( \Phi \left( F\left( X_{2}\right) \right) ,s_{F\left( X_{2}\right)
}^{\Phi }\left( T_{2}\right) \right) 
\end{array}%
\right) \in 
\end{equation*}%
\begin{equation*}
\mathrm{Mor}_{\mathfrak{Cor}_{H_{2}}}\left( \Psi _{\mathrm{Ob}}\left(
F\left( X_{1}\right) /T_{1}\right) ,\Psi _{\mathrm{Ob}}\left( F\left(
X_{2}\right) /T_{2}\right) \right) .
\end{equation*}%
And we define%
\begin{equation}
\Psi _{\mathrm{Mor}}\left( \varphi \right) =\left( \Phi \left( \mu \right) ,%
\begin{array}{c}
\left( \Phi \left( F\left( X_{1}\right) \right) ,s_{F\left( X_{1}\right)
}^{\Phi }\left( T_{1}\right) \right)  \\ 
\left( \Phi \left( F\left( X_{2}\right) \right) ,s_{F\left( X_{2}\right)
}^{\Phi }\left( T_{2}\right) \right) 
\end{array}%
\right)   \label{psi_mor_in}
\end{equation}%
for every $\mu \in \left( 
\begin{array}{c}
\left( F\left( X_{1}\right) ,T_{1}\right)  \\ 
\left( F\left( X_{2}\right) ,T_{2}\right) 
\end{array}%
,\varphi \right) $. In particular it means that%
\begin{equation}
\widetilde{\tau }_{2}\Phi \left( \mu \right) =\Psi _{\mathrm{Mor}}\left(
\varphi \right) \widetilde{\tau }_{1},  \label{psi_mor}
\end{equation}%
where $\widetilde{\tau }_{i}:\Phi \left( F\left( X_{i}\right) \right)
\rightarrow \Phi \left( F\left( X_{i}\right) \right) /s_{F\left(
X_{i}\right) }^{\Phi }\left( T_{i}\right) $, $i=1,2$, are natural
epimorphisms.

\begin{proposition}
\label{psi_functor}$\Psi $ is a functor.
\end{proposition}

\begin{proof}
We consider the object $F\left( X\right) /T\in \mathrm{Ob}\mathfrak{Cor}%
_{H_{1}}$ and its identity morphism $id_{F\left( X\right) /T}$. The diagram%
\begin{equation*}
\begin{array}{ccc}
F\left( X\right)  & \overrightarrow{\tau } & F\left( X\right) /T \\ 
\downarrow id_{F\left( X\right) } &  & id_{F\left( X\right) /T}\downarrow 
\\ 
F\left( X\right)  & \underrightarrow{\tau } & F\left( X\right) /T%
\end{array}%
,
\end{equation*}%
where $\tau $ is a natural epimorphisms, can be closed commutative by
homomorphism $id_{F\left( X\right) }$, because $\tau id_{F\left( X\right)
}=\tau =id_{F\left( X\right) /T}\tau $, so\linebreak $id_{F\left( X\right)
}\in \left( 
\begin{array}{c}
\left( F\left( X\right) ,T\right)  \\ 
\left( F\left( X\right) ,T\right) 
\end{array}%
,id_{F\left( X\right) /T}\right) $. Also the diagram%
\begin{equation*}
\begin{array}{ccc}
\Phi \left( F\left( X\right) \right)  & \overrightarrow{\widetilde{\tau }} & 
\Phi \left( F\left( X\right) \right) /s_{F\left( X\right) }^{\Phi }\left(
T\right)  \\ 
\Phi \left( id_{F\left( X\right) }\right) \downarrow  &  & \downarrow
id_{\Phi \left( F\left( X\right) \right) /s_{F\left( X\right) }^{\Phi
}\left( T\right) } \\ 
\Phi \left( F\left( X\right) \right)  & \overrightarrow{\widetilde{\tau }} & 
\Phi \left( F\left( X\right) \right) /s_{F\left( X\right) }^{\Phi }\left(
T\right) 
\end{array}%
,
\end{equation*}%
where $\widetilde{\tau }$ is a natural epimorphisms commutative. So%
\begin{equation*}
id_{\Psi _{\mathrm{Ob}}\left( F\left( X\right) /T\right) }=id_{\Phi \left(
F\left( X\right) \right) /s_{F\left( X\right) }^{\Phi }\left( T\right) }=
\end{equation*}%
\begin{equation*}
\left( \Phi \left( id_{F\left( X\right) }\right) ,%
\begin{array}{c}
\left( \Phi \left( F\left( X\right) \right) ,s_{F\left( X\right) }^{\Phi
}\left( T\right) \right)  \\ 
\left( \Phi \left( F\left( X\right) \right) ,s_{F\left( X\right) }^{\Phi
}\left( T\right) \right) 
\end{array}%
\right) =\Psi _{\mathrm{Mor}}\left( id_{F\left( X\right) /T}\right) 
\end{equation*}%
by (\ref{psi_ob}) and (\ref{psi_mor_in}).

Now we consider $F\left( X_{i}\right) /T_{i}\in \mathrm{Ob}\mathfrak{Cor}%
_{H_{1}}$, $i=1,2,3$, and%
\begin{equation*}
\varphi _{i}\in \mathrm{Mor}_{\mathfrak{Cor}_{H_{1}}}\left( F\left(
X_{i}\right) /T_{i},F\left( X_{i+1}\right) /T_{i+1}\right) ,
\end{equation*}%
$i=1,2$. By Proposition \ref{projective} there exist $\mu _{i}\in \left( 
\begin{array}{c}
\left( F\left( X_{i}\right) ,T_{i}\right)  \\ 
\left( F\left( X_{i+1}\right) ,T_{i+1}\right) 
\end{array}%
,\varphi _{i}\right) $, where $i=1,2$. Therefore, $\mu _{2}\mu _{1}\in
\left( 
\begin{array}{c}
\left( F\left( X_{1}\right) ,T_{1}\right)  \\ 
\left( F\left( X_{3}\right) ,T_{3}\right) 
\end{array}%
,\varphi _{2}\varphi _{1}\right) $ by Remark \ref{nn_trans}. So,%
\begin{equation*}
\Psi _{\mathrm{Mor}}\left( \varphi _{2}\varphi _{1}\right) =\left( \Phi
\left( \mu _{2}\mu _{1}\right) ,%
\begin{array}{c}
\left( \Phi \left( F\left( X_{1}\right) \right) ,s_{F\left( X_{1}\right)
}^{\Phi }\left( T_{1}\right) \right)  \\ 
\left( \Phi \left( F\left( X_{3}\right) \right) ,s_{F\left( X_{3}\right)
}^{\Phi }\left( T_{3}\right) \right) 
\end{array}%
\right) 
\end{equation*}%
by (\ref{psi_mor_in}). Also%
\begin{equation*}
\Psi _{\mathrm{Mor}}\left( \varphi _{i}\right) =\left( \Phi \left( \mu
_{i}\right) ,%
\begin{array}{c}
\left( \Phi \left( F\left( X_{i}\right) \right) ,s_{F\left( X_{i}\right)
}^{\Phi }\left( T_{i}\right) \right)  \\ 
\left( \Phi \left( F\left( X_{i+1}\right) \right) ,s_{F\left( X_{i+1}\right)
}^{\Phi }\left( T_{i+1}\right) \right) 
\end{array}%
\right) ,
\end{equation*}%
where $i=1,2$, by (\ref{psi_mor_in}). Therefore,%
\begin{equation*}
\Psi _{\mathrm{Mor}}\left( \varphi _{2}\right) \Psi _{\mathrm{Mor}}\left(
\varphi _{1}\right) \in \left( \Phi \left( \mu _{2}\right) \Phi \left( \mu
_{1}\right) ,%
\begin{array}{c}
\left( \Phi \left( F\left( X_{1}\right) \right) ,s_{F\left( X_{1}\right)
}^{\Phi }\left( T_{1}\right) \right)  \\ 
\left( \Phi \left( F\left( X_{3}\right) \right) ,s_{F\left( X_{i+1}\right)
}^{\Phi }\left( T_{3}\right) \right) 
\end{array}%
\right) =
\end{equation*}%
\begin{equation*}
\left( \Phi \left( \mu _{2}\mu _{1}\right) ,%
\begin{array}{c}
\left( \Phi \left( F\left( X_{1}\right) \right) ,s_{F\left( X_{1}\right)
}^{\Phi }\left( T_{1}\right) \right)  \\ 
\left( \Phi \left( F\left( X_{3}\right) \right) ,s_{F\left( X_{i+1}\right)
}^{\Phi }\left( T_{3}\right) \right) 
\end{array}%
\right) 
\end{equation*}%
also by Remark \ref{nn_trans}. Hence $\Psi _{\mathrm{Mor}}\left( \varphi
_{2}\varphi _{1}\right) =\Psi _{\mathrm{Mor}}\left( \varphi _{2}\right) \Psi
_{\mathrm{Mor}}\left( \varphi _{1}\right) $.
\end{proof}

\begin{proposition}
\label{right_comm}The functor $\Psi $ closes the right diagram in Definition %
\ref{new} commutative.
\end{proposition}

\begin{proof}
We consider $\left( F\left( X\right) ,T\right) \in \mathrm{Ob}\mathfrak{Cl}%
_{H_{1}}$. We have, by (\ref{delta_ob}), (\ref{FR_ob}) and (\ref{psi_ob}),
that%
\begin{equation*}
\mathcal{FR}_{H_{2}}\left( \Lambda \left( F\left( X\right) ,T\right) \right)
=\mathcal{FR}_{H_{2}}\left( \Phi (F\left( X\right) ),s_{F\left( X\right)
}^{\Phi }\left( T\right) \right) =\Phi (F\left( X\right) )/s_{F\left(
X\right) }^{\Phi }\left( T\right)
\end{equation*}%
and%
\begin{equation*}
\Psi \left( \mathcal{FR}_{H_{1}}\left( F\left( X\right) ,T\right) \right)
=\Psi \left( F\left( X\right) /T\right) =\Phi (F\left( X\right) )/s_{F\left(
X\right) }^{\Phi }\left( T\right) .
\end{equation*}

Now we consider $\left( F\left( X_{i}\right) ,T_{i}\right) \in \mathrm{Ob}%
\mathfrak{Cl}_{H_{1}}$, $i=1,2$, and%
\begin{equation*}
\mu \in \mathrm{Mor}_{\mathfrak{Cl}_{H_{1}}}\left( \left( F\left(
X_{1}\right) ,T_{1}\right) ,\left( F\left( X_{2}\right) ,T_{2}\right)
\right) .
\end{equation*}%
It means that $T_{i}\in Cl_{H_{1}}\left( F\left( X_{i}\right) \right) $, $%
\mu \in \mathrm{Hom}\left( F\left( X_{1}\right) ,F\left( X_{2}\right)
\right) $ and $\mu \left( T_{1}\right) \subseteq T_{2}$. We have by (\ref%
{FR_mor_s}) that%
\begin{equation*}
\mathcal{FR}_{H_{1}}\left( \mu \right) =\left( \mu ,%
\begin{array}{c}
\left( F\left( X_{1}\right) ,T_{1}\right)  \\ 
\left( F\left( X_{2}\right) ,T_{2}\right) 
\end{array}%
\right) .
\end{equation*}%
By Remark \ref{nn_sym}, we have that $\mu \in \left( 
\begin{array}{c}
\left( F\left( X_{1}\right) ,T_{1}\right)  \\ 
\left( F\left( X_{2}\right) ,T_{2}\right) 
\end{array}%
,\mathcal{FR}_{H_{1}}\left( \mu \right) \right) $. So,%
\begin{equation*}
\Psi \left( \mathcal{FR}_{H_{1}}\left( \mu \right) \right) =\left( \Phi
\left( \mu \right) ,%
\begin{array}{c}
\left( \Phi \left( F\left( X_{1}\right) \right) ,s_{F\left( X_{1}\right)
}^{\Phi }\left( T_{1}\right) \right)  \\ 
\left( \Phi \left( F\left( X_{2}\right) \right) ,s_{F\left( X_{2}\right)
}^{\Phi }\left( T_{2}\right) \right) 
\end{array}%
\right) 
\end{equation*}%
by (\ref{psi_mor_in}).

Also we have that $\Phi \left( \mu \right) \in \mathrm{Hom}\left( \Phi
\left( F\left( X_{1}\right) \right) ,\Phi \left( F\left( X_{2}\right)
\right) \right) $,\linebreak $s_{F\left( X_{i}\right) }^{\Phi }\left(
T_{i}\right) \in Cl_{H_{2}}\left( \Phi \left( F\left( X_{i}\right) \right)
\right) $ by (\ref{alpha}), where $i=1,2$, and $\Phi \left( \mu \right)
\left( s_{F\left( X_{1}\right) }^{\Phi }\left( T_{1}\right) \right)
\subseteq s_{F\left( X_{2}\right) }^{\Phi }\left( T_{2}\right) $ by
Proposition \ref{prop_phi_phi_t_incl}. Therefore $\left( \Phi \left( F\left(
X_{i}\right) \right) ,s_{F\left( X_{i}\right) }^{\Phi }\left( T_{i}\right)
\right) \in \mathrm{Ob}\mathfrak{Cl}_{H_{2}}$, $i=1,2$, and%
\begin{equation*}
\Phi \left( \mu \right) \in \mathrm{Mor}_{\mathfrak{Cl}_{H_{2}}}\left(
\left( \Phi (F\left( X_{1}\right) ),s_{F\left( X_{1}\right) }^{\Phi }\left(
T_{1}\right) \right) ,\left( \Phi (F\left( X_{2}\right) ),s_{F\left(
X_{2}\right) }^{\Phi }\left( T_{2}\right) \right) \right) .
\end{equation*}%
\begin{equation*}
\left( \mathcal{FR}_{H_{2}}\right) \left( \Lambda \left( \mu \right) \right)
=\left( \mathcal{FR}_{H_{2}}\right) \left( \Phi \left( \mu \right) \right)
=\left( \Phi \left( \mu \right) ,%
\begin{array}{c}
\left( \Phi (F\left( X_{1}\right) ),s_{F\left( X_{1}\right) }^{\Phi }\left(
T_{1}\right) \right)  \\ 
\left( \Phi (F\left( X_{2}\right) ),s_{F\left( X_{2}\right) }^{\Phi }\left(
T_{2}\right) \right) 
\end{array}%
\right) ,
\end{equation*}%
by (\ref{lambda_mor}) and (\ref{FR_mor_s}). Therefore $\Psi \left( \mathcal{%
FR}_{H_{1}}\left( \mu \right) \right) =\left( \mathcal{FR}_{H_{2}}\right)
\left( \Lambda \left( \mu \right) \right) $.
\end{proof}

\begin{proposition}
\label{psi_isom}$\Psi $ is an isomorphism of categories.
\end{proposition}

\begin{proof}
We remind again that in our assumptions automorphism $\Phi ^{-1}:\Theta
^{0}\rightarrow \Theta ^{0}$ provides the automorphic equivalence of
algebras $H_{2}$ and $H_{1}$ by Remark \ref{phi_inv_eq}. Hence exists a
functor $\Lambda ^{\prime }$ which is also defined by (\ref{delta_ob}) and (%
\ref{lambda_mor}) with respect to the automorphism $\Phi ^{-1}$. By
Propositions \ref{psi_functor} and \ref{right_comm}, there exists a functor $%
\Psi ^{\prime }$ which is also defined by (\ref{psi_ob}) and (\ref%
{psi_mor_in}) with respect to the automorphism $\Phi ^{-1}$. This functor
closes the diagram%
\begin{equation}
\begin{array}{ccc}
\mathfrak{Cl}_{H_{2}} & \overrightarrow{\Lambda ^{\prime }} & \mathfrak{Cl}%
_{H_{1}} \\ 
\downarrow \mathcal{FR}_{H_{2}} &  & \mathcal{FR}_{H_{1}}\downarrow \\ 
\mathfrak{Cor}_{H_{2}} & \underrightarrow{\Psi ^{\prime }} & \mathfrak{Cor}%
_{H_{1}}%
\end{array}
\label{right_new_inv}
\end{equation}%
commutative. Therefore we have by Proposition \ref{right_comm} and by (\ref%
{lambda_inv_1}) that%
\begin{equation*}
\Psi ^{\prime }\Psi \mathcal{FR}_{H_{1}}=\Psi ^{\prime }\mathcal{FR}%
_{H_{2}}\Lambda =\mathcal{FR}_{H_{1}}\Lambda ^{\prime }\Lambda =\mathcal{FR}%
_{H_{1}}id_{\mathfrak{Cl}_{H_{1}}}=id_{\mathfrak{Cor}_{H_{1}}}\mathcal{FR}%
_{H_{1}},
\end{equation*}%
so%
\begin{equation*}
\Psi ^{\prime }\Psi =id_{\mathfrak{Cor}_{H_{1}}}
\end{equation*}%
by Proposition \ref{FR_epi} and Remark \ref{epi_f}. Symmetrical reasoning
gives us that%
\begin{equation*}
\Psi \Psi ^{\prime }=id_{\mathfrak{Cor}_{H_{2}}}.
\end{equation*}
\end{proof}

Now, we can conclude from Propositions \ref{lambda_functor}, \ref{left_comm}%
, \ref{lambda_isom}, \ref{psi_functor}, \ref{right_comm}, \ref{psi_isom} that

\begin{proposition}
\label{old_>new}If two algebras $H_{1},H_{2}\in \Theta $ are subjects of
Definition \ref{old} then they are subjects of Definition \ref{new}.
\end{proposition}

\subsection{Consequences of the Definition \protect\ref{new}}

\subsubsection{Consequences of the commutativity of the left diagram from
the Definition \protect\ref{new}}

\begin{proposition}
\label{left_diagram}If for two algebras $H_{1},H_{2}\in \Theta $ exists an
isomorphism $\Lambda $, which closes the left diagram of the Definition \ref%
{new} commutative, then the%
\begin{equation}
\Lambda \left( F\left( X\right) ,T\right) =\left( \Phi \left( F\left(
X\right) \right) ,\alpha (\Phi )_{F\left( X\right) }\left( T\right) \right)
\label{lambda_new}
\end{equation}%
holds for every $\left( F\left( X\right) ,T\right) \in \mathrm{Ob}\mathfrak{%
Cl}_{H_{1}}$, where%
\begin{equation}
\alpha (\Phi )_{F\left( X\right) }:Cl_{H_{1}}(F\left( X\right) )\rightarrow
Cl_{H_{2}}(\Phi (F\left( X\right) ))  \label{alfa_Phi_F}
\end{equation}%
is a \textit{bijection}.
\end{proposition}

\begin{proof}
We have that $T\in Cl_{H_{1}}F\left( X\right) $ by definition of $\mathfrak{%
Cl}_{H_{1}}$. $\Lambda \left( F\left( X\right) ,T\right) =\left( F\left(
Y\right) ,\widetilde{T}\right) \in \mathrm{Ob}\mathfrak{Cl}_{H_{2}}$. It
means that $F\left( Y\right) \in \mathrm{Ob}\Theta ^{0}$, $\widetilde{T}\in
Cl_{H_{2}}F\left( Y\right) $. We can conclude from our conditions and from (%
\ref{FG_ob}) that%
\begin{equation*}
\mathcal{FG}_{H_{2}}\left( \Lambda \left( F\left( X\right) ,T\right) \right)
=\mathcal{FG}_{H_{2}}\left( F\left( Y\right) ,\widetilde{T}\right) =F\left(
Y\right) =
\end{equation*}%
\begin{equation*}
\Phi \left( \mathcal{FG}_{H_{1}}\left( F\left( X\right) ,T\right) \right)
=\Phi \left( F\left( X\right) \right) .
\end{equation*}

We denote $\alpha (\Phi )_{F\left( X\right) }\left( T\right) =\widetilde{T}$%
. We fix $F\left( X\right) \in \mathrm{Ob}\Theta ^{0}$. $\left( F\left(
X\right) ,T\right) \in \mathrm{Ob}\mathfrak{Cl}_{H_{1}}$ for every $T\in
Cl_{H_{1}}F\left( X\right) $, so $\alpha (\Phi )_{F\left( X\right)
}:Cl_{H_{1}}(F\left( X\right) )\rightarrow Cl_{H_{2}}(\Phi (F\left( X\right)
))$ is a mapping. Now we will prove that $\alpha (\Phi )_{F\left( X\right) }$
is a bijection. We have by our condition that%
\begin{equation*}
\Phi ^{-1}\mathcal{FG}_{H_{2}}=\Phi ^{-1}\mathcal{FG}_{H_{2}}\Lambda \Lambda
^{-1}=\Phi ^{-1}\Phi \mathcal{FG}_{H_{1}}\Lambda ^{-1}=\mathcal{FG}%
_{H_{1}}\Lambda ^{-1}.
\end{equation*}%
It means that for algebras $H_{1},H_{2}$ the left diagram of the Definition %
\ref{new} are closed commutative with isomorphisms $\Lambda ^{-1}$ and $\Phi
^{-1}$. So, as above, for every $F\left( Y\right) \in \mathrm{Ob}\Theta ^{0}$
there exists a mapping $\alpha (\Phi ^{-1})_{F\left( Y\right)
}:Cl_{H_{2}}(F\left( Y\right) )\rightarrow Cl_{H_{1}}(\Phi ^{-1}(F\left(
Y\right) ))$ such that $\Lambda ^{-1}\left( F\left( Y\right) ,R\right)
=\left( \Phi ^{-1}\left( F\left( Y\right) \right) ,\alpha (\Phi
^{-1})_{F\left( Y\right) }\left( R\right) \right) $ holds for every $R\in
Cl_{H_{2}}(F\left( Y\right) )$. We consider $\left( F\left( X\right)
,T\right) \in \mathrm{Ob}\mathfrak{Cl}_{H_{1}}$ and we denote $\Phi \left(
F\left( X\right) \right) =F\left( Y\right) $. Therefore%
\begin{equation*}
\left( F\left( X\right) ,T\right) =\Lambda ^{-1}\Lambda \left( F\left(
X\right) ,T\right) =\Lambda ^{-1}\left( \Phi \left( F\left( X\right) \right)
,\alpha (\Phi )_{F\left( X\right) }\left( T\right) \right) =
\end{equation*}%
\begin{equation*}
\left( \Phi ^{-1}\Phi \left( F\left( X\right) \right) ,\alpha (\Phi
^{-1})_{F\left( Y\right) }\alpha (\Phi )_{F\left( X\right) }\left( T\right)
\right) .
\end{equation*}%
Hence, $\alpha (\Phi ^{-1})_{F\left( Y\right) }\alpha (\Phi )_{F\left(
X\right) }\left( T\right) =T$ for every $T\in Cl_{H_{1}}(F\left( X\right) )$%
. On the other hand, we have for every $\left( F\left( Y\right) ,R\right)
\in \mathrm{Ob}\mathfrak{Cl}_{H_{2}}$ that%
\begin{equation*}
\left( F\left( Y\right) ,R\right) =\Lambda \Lambda ^{-1}\left( F\left(
Y\right) ,R\right) =\Lambda \left( F\left( X\right) ,\alpha (\Phi
^{-1})_{F\left( Y\right) }\left( R\right) \right) =
\end{equation*}%
\begin{equation*}
\Lambda \left( F\left( X\right) ,T\right) =\left( \Phi \left( F\left(
X\right) \right) ,\alpha (\Phi )_{F\left( X\right) }\alpha (\Phi
^{-1})_{F\left( Y\right) }\left( R\right) \right) ,
\end{equation*}%
because $F\left( X\right) =\Phi ^{-1}\left( F\left( Y\right) \right) $. It
means that $\alpha (\Phi )_{F\left( X\right) }\alpha (\Phi ^{-1})_{F\left(
Y\right) }\left( R\right) =R$ for every $R\in Cl_{H_{2}}(F\left( Y\right) )$%
. Therefore, the mappings $\alpha (\Phi )_{F\left( X\right)
}:Cl_{H_{1}}(F\left( X\right) )\rightarrow Cl_{H_{2}}(F\left( Y\right) )$
and $\alpha (\Phi ^{-1})_{F\left( Y\right) }:Cl_{H_{2}}(F\left( Y\right)
)\rightarrow Cl_{H_{1}}(F\left( X\right) )$ are inverse to each other. So $%
\alpha (\Phi )_{F\left( X\right) }$ is a bijection.
\end{proof}

\begin{proposition}
\label{left_diagram_mor}If for two algebras $H_{1},H_{2}\in \Theta $ exists
an isomorphism $\Lambda $, which closes the left diagram of the Definition %
\ref{new} commutative, then $\Lambda \left( \varphi \right) =\Phi \left(
\varphi \right) $ holds for every $\varphi \in \mathrm{Mor}_{\mathfrak{Cl}%
_{H_{1}}}\left( \left( F\left( X_{1}\right) ,T_{1}\right) ,\left( F\left(
X_{2}\right) ,T_{2}\right) \right) $, where\linebreak $\left( F\left(
X_{1}\right) ,T_{1}\right) ,\left( F\left( X_{2}\right) ,T_{2}\right) \in 
\mathrm{Ob}\mathfrak{Cl}_{H_{1}}$.
\end{proposition}

\begin{proof}
The%
\begin{equation}
\mathcal{FG}_{H_{2}}\left( \Lambda \left( \varphi \right) \right) =\Lambda
\left( \varphi \right) =\Phi \left( \mathcal{FG}_{H_{1}}\left( \varphi
\right) \right) =\Phi \left( \varphi \right)  \label{Lambda_phi}
\end{equation}%
holds by (\ref{FG_mor}).
\end{proof}

\subsubsection{Consequence of the commutativity of two diagrams from the
Definition \protect\ref{new}}

\begin{proposition}
\label{new_>old}If two algebras $H_{1},H_{2}\in \Theta $ are subjects of
Definition \ref{new} then they are subjects of Definition \ref{old}.
\end{proposition}

\begin{proof}
We just proved in Proposition \ref{left_diagram} that in our conditions for
every $F\left( X\right) \in \mathrm{Ob}\Theta ^{0}$ there exists a bijection%
\begin{equation*}
\alpha (\Phi )_{F\left( X\right) }:Cl_{H_{1}}(F\left( X\right) )\rightarrow
Cl_{H_{2}}(\Phi (F\left( X\right) )).
\end{equation*}

Now we consider $F\left( X_{1}\right) ,F\left( X_{2}\right) \in \mathrm{Ob}%
\Theta ^{0}$\textit{, }$\mu _{1},\mu _{2}\in \mathrm{Hom}\left( F\left(
X_{1}\right) ,F\left( X_{2}\right) \right) $\textit{, }$T\in
Cl_{H_{1}}(F\left( X_{2}\right) )$.%
\begin{equation}
\mu _{1},\mu _{2}\in \mathrm{Mor}_{\mathfrak{Cl}_{H_{1}}}\left( \left(
F\left( X_{1}\right) ,\left( \Delta _{F\left( X_{1}\right) }\right)
_{H_{1}}^{\prime \prime }\right) ,\left( F\left( X_{2}\right) ,T\right)
\right)  \label{mu12_incl}
\end{equation}%
holds by Proposition \ref{allTheta_0_morpfism}. So,%
\begin{equation}
\mathcal{FR}_{H_{1}}\left( \mu _{i}\right) =\left( \mu _{i},%
\begin{array}{c}
\left( F\left( X_{1}\right) ,\left( \Delta _{F\left( X_{1}\right) }\right)
_{H_{1}}^{\prime \prime }\right) \\ 
\left( F\left( X_{2}\right) ,T\right)%
\end{array}%
\right) ,  \label{FR_mu}
\end{equation}%
where $i=1,2$, holds by (\ref{FR_mor_s}). Also we have by (\ref{alfa_Phi_F}%
), (\ref{mu12_incl}), (\ref{Lambda_phi}) and by Propositions \ref%
{left_diagram}, \ref{left_diagram_mor} that%
\begin{equation*}
\Phi \left( \mu _{i}\right) =\Lambda \left( \mu _{i}\right) \in \mathrm{Mor}%
_{\mathfrak{Cl}_{H_{2}}}\left( \Lambda \left( F\left( X_{1}\right) ,\left(
\Delta _{F\left( X_{1}\right) }\right) _{H_{1}}^{\prime \prime }\right)
,\Lambda \left( F\left( X_{2}\right) ,T\right) \right) =
\end{equation*}%
\begin{equation*}
\mathrm{Mor}_{\mathfrak{Cl}_{H_{2}}}\left( \left( \Phi \left( F\left(
X_{1}\right) \right) ,\alpha (\Phi )_{F\left( X_{1}\right) }\left( \left(
\Delta _{F\left( X_{1}\right) }\right) _{H_{1}}^{\prime \prime }\right)
\right) ,\left( \Phi \left( F\left( X_{2}\right) \right) ,\alpha (\Phi
)_{F\left( X_{2}\right) }\left( T\right) \right) \right)
\end{equation*}%
for $i=1,2$. So, we have by (\ref{FR_mor_s})%
\begin{equation}
\mathcal{FR}_{H_{2}}\left( \Phi \left( \mu _{i}\right) \right) =\left( \Phi
\left( \mu _{i}\right) ,%
\begin{array}{c}
\left( \Phi \left( F\left( X_{1}\right) \right) ,\alpha (\Phi )_{F\left(
X_{1}\right) }\left( \left( \Delta _{F\left( X_{1}\right) }\right)
_{H_{1}}^{\prime \prime }\right) \right) \\ 
\left( \Phi \left( F\left( X_{2}\right) \right) ,\alpha (\Phi )_{F\left(
X_{2}\right) }\left( T\right) \right)%
\end{array}%
\right) ,  \label{FR_H_2_Phi}
\end{equation}%
for $i=1,2$. The right diagram of the Definition \ref{new} is commutative,
so, by (\ref{Lambda_phi}),%
\begin{equation}
\mathcal{FR}_{H_{2}}\left( \Phi \left( \mu _{i}\right) \right) =\mathcal{FR}%
_{H_{2}}\left( \Lambda \left( \mu _{i}\right) \right) =\Psi \mathcal{FR}%
_{H_{1}}\left( \mu _{i}\right)  \label{FR_H_2}
\end{equation}%
holds for $i=1,2$.

We consider these natural epimorphisms: $\delta :F\left( X_{1}\right)
\rightarrow F\left( X_{1}\right) /\left( \Delta _{F\left( X_{1}\right)
}\right) _{H_{1}}^{\prime \prime }$, $\tau :F\left( X_{2}\right) \rightarrow
F\left( X_{2}\right) /T$, $\widetilde{\delta }:\Phi \left( F\left(
X_{1}\right) \right) \rightarrow \Phi \left( F\left( X_{1}\right) \right)
/\alpha (\Phi )_{F\left( X_{1}\right) }\left( \left( \Delta _{F\left(
X_{1}\right) }\right) _{H_{1}}^{\prime \prime }\right) $, $\widetilde{\tau }%
:\Phi \left( F\left( X_{2}\right) \right) \rightarrow \Phi \left( F\left(
X_{2}\right) \right) /\alpha (\Phi )_{F\left( X_{2}\right) }\left( T\right) $%
.

We suppose that (\ref{old_1}) holds. We have from (\ref{FR_mu}) that%
\begin{equation}
\tau \mu _{1}=\mathcal{FR}_{H_{1}}\left( \mu _{1}\right) \delta =\tau \mu
_{2}=\mathcal{FR}_{H_{1}}\left( \mu _{2}\right) \delta ,  \label{eqv}
\end{equation}%
so%
\begin{equation}
\mathcal{FR}_{H_{1}}\left( \mu _{1}\right) =\mathcal{FR}_{H_{1}}\left( \mu
_{2}\right)  \label{FR_mu_H1_eqv}
\end{equation}%
by Remark \ref{epi}, and%
\begin{equation}
\mathcal{FR}_{H_{2}}\left( \Phi \left( \mu _{1}\right) \right) =\mathcal{FR}%
_{H_{2}}\left( \Phi \left( \mu _{2}\right) \right)  \label{FR_phi_mu_H2_eqv}
\end{equation}%
by (\ref{FR_H_2}). Therefore, we have by (\ref{FR_H_2_Phi}) that%
\begin{equation}
\widetilde{\tau }\Phi \left( \mu _{1}\right) =\mathcal{FR}_{H_{2}}\left(
\Phi \left( \mu _{1}\right) \right) \widetilde{\delta }=\mathcal{FR}%
_{H_{2}}\left( \Phi \left( \mu _{2}\right) \right) \widetilde{\delta }=%
\widetilde{\tau }\Phi \left( \mu _{2}\right) .  \label{Chil_eqv}
\end{equation}%
It means that (\ref{old_2}) holds.

Now we suppose that (\ref{old_2}) holds. We conclude (\ref{Chil_eqv}) from (%
\ref{old_2}) and (\ref{FR_H_2_Phi}). So, we have (\ref{FR_phi_mu_H2_eqv}) by
Remark \ref{epi}. We can conclude (\ref{FR_mu_H1_eqv}) from (\ref%
{FR_phi_mu_H2_eqv}) and (\ref{FR_H_2}), because $\Psi $ is an isomorphism.
And we conclude (\ref{eqv}) from (\ref{FR_mu_H1_eqv}) and (\ref{FR_mu}). It
means that (\ref{old_1}) holds. This completes the proof.
\end{proof}

Now we can conclude from Propositions \ref{old_>new} and \ref{new_>old} the

\begin{theorem}
\label{main}Two algebras $H_{1},H_{2}\in \Theta $ are subjects of Definition %
\ref{new} if and only if they are subjects of Definition \ref{old}.
\end{theorem}

\section{First corollaries}

\setcounter{equation}{0}

We have from Theorem \ref{main} the

\begin{corollary}
If two algebras $H_{1},H_{2}\in \Theta $ are subjects of Definition \ref{new}%
, then isomorphisms $\Lambda \ $and $\Psi $ are uniquely defined by
automorphism $\Phi $.
\end{corollary}

\begin{proof}
If two algebras $H_{1},H_{2}\in \Theta $ are subjects of Definition \ref{new}
then, the $\Lambda \left( F\left( X\right) ,T\right) =\left( \Phi \left(
F\left( X\right) \right) ,\alpha (\Phi )_{F\left( X\right) }\left( T\right)
\right) $ holds for every $\left( F\left( X\right) ,T\right) \in \mathrm{Ob}%
\mathfrak{Cl}_{H_{1}}$ by (\ref{lambda_new}). Algebras $H_{1},H_{2}\in
\Theta $ are also subjects of Definition \ref{old} by Theorem \ref{main}.
So, the bijections%
\begin{equation*}
\alpha (\Phi )_{F\left( X\right) }:Cl_{H_{1}}(F\left( X\right) )\rightarrow
Cl_{H_{2}}(\Phi (F\left( X\right) ))
\end{equation*}%
for every $F\left( X\right) \in \mathrm{Ob}\Theta ^{0}$ are uniquely defined
by automorphism $\Phi $ by Remark \ref{alpha_unic}. Therefore $\Lambda
\left( F\left( X\right) ,T\right) $ is uniquely defined by automorphism $%
\Phi $. Also, by (\ref{Lambda_phi}), $\Lambda \left( \varphi \right) $
uniquely defined by automorphism $\Phi $ for every%
\begin{equation*}
\left( F\left( X_{1}\right) ,T_{1}\right) ,\left( F\left( X_{2}\right)
,T_{2}\right) \in \mathrm{Ob}\mathfrak{Cl}_{H_{1}}
\end{equation*}%
and every%
\begin{equation*}
\varphi \in \mathrm{Mor}_{\mathfrak{Cl}_{H_{1}}}\left( \left( F\left(
X_{1}\right) ,T_{1}\right) ,\left( F\left( X_{2}\right) ,T_{2}\right)
\right) .
\end{equation*}%
So, the isomorphism $\Lambda $ uniquely defined by automorphism $\Phi $.

We suppose that $\Psi _{1},\Psi _{2}:\mathfrak{Cor}_{H_{1}}\rightarrow 
\mathfrak{Cor}_{H_{2}}$ are two isomorphisms which close commutative the
right diagram of Definition \ref{new}. It means that the%
\begin{equation}
\Psi _{1}\mathcal{FR}_{H_{1}}=\mathcal{FR}_{H_{2}}\Lambda =\Psi _{2}\mathcal{%
FR}_{H_{1}}  \label{left_diagr_commut_1_2}
\end{equation}%
So, $\Psi _{1}=\Psi _{2}$ by Proposition \ref{FR_epi} and Remark \ref{epi_f}%
. Therefore, the isomorphism $\Psi $ which closes commutative the right
diagram of Definition \ref{new} uniquely defined by isomorphism $\Lambda $,
which is uniquely defined by automorphism $\Phi $.
\end{proof}

This Corollary justifies the name: automorphic equivalence. And it
justifies, if we talk about two algebras $H_{1},H_{2}\in \Theta $, which are
subjects of Definition \ref{new}, that automorphic equivalence of this
algebras provided by automorphism\textit{\ }$\Phi $ of the category $\Theta
^{0}$.

The new definition allows us to more easily prove the following

\begin{proposition}
The identity automorphism\textit{\ }of the category $\Theta ^{0}$ provides
the automorphic equivalence of every algebra $H\in \Theta $ to itself. If
automorphic equivalence of two algebras $H_{1}\ $and $H_{2}$ from variety $%
\Theta $ provided by automorphism\textit{\ }$\Phi $ of the category $\Theta
^{0}$, then the automorphism\textit{\ }$\Phi ^{-1}$ provides the automorphic
equivalence of two algebras $H_{2}$\ and $H_{1}$. If automorphisms\textit{\ }%
$\Phi ,\Psi $ of the category $\Theta ^{0}$ provide the automorphic
equivalence of algebras $H_{1}\ $and $H_{2}$ and algebras $H_{2}\ $and $%
H_{3} $ respectively then the automorphism\textit{\ }$\Psi \Phi $ provides
the automorphic equivalence of algebras $H_{1}$\ and $H_{3}$.
\end{proposition}

To prove this proposition, it suffices to consider the commutative diagrams
from Definition \ref{new}.

We have from this proposition the

\begin{corollary}
Automorphic equivalence of algebras in $\Theta $ is an equivalence in the
class of all the algebras of the variety $\Theta $.
\end{corollary}

\begin{definition}
\label{inner}An automorphism $\Upsilon $ of an arbitrary category $\mathfrak{%
K}$ is \textbf{inner}, if it is isomorphic as a functor to the identity
automorphism of the category $\mathfrak{K}$.
\end{definition}

It means that there exists a system of isomorphisms%
\begin{equation*}
\left\{ \sigma _{F\left( X\right) }^{\Upsilon }:F\left( X\right) \rightarrow
\Upsilon \left( F\left( X\right) \right) \mid F\left( X\right) \in \mathrm{Ob%
}\mathfrak{K}\right\} 
\end{equation*}%
such that for every $\mu \in \mathrm{Mor}_{\mathfrak{K}}\left( F\left(
X_{1}\right) ,F\left( X_{2}\right) \right) $ the diagram%
\begin{equation*}
\begin{array}{ccc}
F\left( X_{1}\right)  & \overrightarrow{\sigma _{F\left( X_{1}\right)
}^{\Upsilon }} & \Upsilon \left( F\left( X_{1}\right) \right)  \\ 
\downarrow \mu  &  & \Upsilon \left( \mu \right) \downarrow  \\ 
F\left( X_{2}\right)  & \underrightarrow{\sigma _{F\left( X_{2}\right)
}^{\Upsilon }} & \Upsilon \left( F\left( X_{2}\right) \right) 
\end{array}%
\end{equation*}%
\noindent commutes. If we take $\mathfrak{K}=\Theta ^{0}$ and compare the
definition of inner automorphism with condition (\ref{autom_action}), which
fulfills for every automorphism of the category $\Theta ^{0}$, we will find
that for an inner automorphism the role of a system of bijections can be
played by a system of isomorphisms. It is easy to prove that the set of all
inner automorphisms of an arbitrary category $\mathfrak{K}$ is a normal
subgroup of the group of all automorphisms of this category.

We know (\cite[Proposition 9]{PlotkinSame}, \cite[Theorem 4.2]%
{TsurkovManySorted}) that if an inner automorphism $\Upsilon $ of the
category $\Theta ^{0}$, where $\Theta $ is some variety of algebras,
provides the automorphic equivalence of the algebras $H_{1}$ and $H_{2}$,
where $H_{1},H_{2}\in \Theta $, then $H_{1}$ and $H_{2}$ are geometrically
equivalent. So the difference between geometric equivalence and automorphic
equivalence in the variety $\Theta $ can exist only when the quotient group $%
\mathfrak{A/Y}$ is not a trivial, where $\mathfrak{A}$ is the group of all
the automorphisms of the category $\Theta ^{0}$ and $\mathfrak{Y}$ is the
group of all the inner automorphisms of this category. For this reason, the
study of this quotient group is so important in universal algebraic geometry.

\section{Method of verbal operations for computation the group $\mathfrak{A/Y%
}$ and some conclusions.}

This method was elaborated in \cite{PlZhit} and refined in \cite%
{TsurkovManySorted}. All definitions and results that we present in this
section with out special citations were published in the two above mentioned
articles.

\begin{definition}
\label{str_st}\textit{An automorphism }$\Phi $\textit{\ of the category }$%
\Theta ^{0}$\textit{\ is called \textbf{strongly stable} if}
\end{definition}

\begin{enumerate}
\item $\Phi \left( F\left( X\right) \right) =F\left( X\right) $\textit{\ for
every }$F\left( X\right) \in \mathrm{Ob}\Theta ^{0}$\textit{\ and }

\item \textit{there exists a system of bijections}%
\begin{equation*}
\left\{ s_{F\left( X\right) }^{\Phi }:F\left( X\right) \rightarrow \Phi
\left( F\left( X\right) \right) \mid F\left( X\right) \in \mathrm{Ob}\Theta
^{0}\right\} 
\end{equation*}%
\textit{such that for every }$\mu \in \mathrm{Hom}\left( F\left(
X_{1}\right) ,F\left( X_{2}\right) \right) $\textit{\ the (\ref{autom_action}%
) holds\ and}%
\begin{equation*}
s_{F\left( X\right) }^{\Phi }\mid _{X}=id_{X},
\end{equation*}%
\textit{for every }$F\left( X\right) \in \mathrm{Ob}\Theta ^{0}$.
\end{enumerate}

The set of all strongly stable automorphisms of the category $\Theta ^{0}$
forms a group. We denote this group by $\mathfrak{S}$.

\begin{theorem}
The $\mathfrak{A=YS}$ holds.
\end{theorem}

Therefore $\mathfrak{A/Y\cong S/S\cap Y}$. So, the most important thing in
the study of the group $\mathfrak{A/Y}$ is the calculation of the group $%
\mathfrak{S}$.

For every word $w=w\left( x_{1},\ldots ,x_{n}\right) \in F\left(
x_{1},\ldots ,x_{n}\right) \in \mathrm{Ob}\Theta ^{0}$ and every $H\in
\Theta $ we can define an operation $w_{H}^{\ast }$ with arity $\rho \left(
w_{H}^{\ast }\right) =n$: if $h_{1},\ldots ,h_{n}\in H$ then $w_{H}^{\ast
}\left( h_{1},\ldots ,h_{n}\right) =w\left( h_{1},\ldots ,h_{n}\right) $.

We denote by $\Omega $ the signature of our variety $\Theta $ and we denote
for every $\omega \in \Omega $ arity of this operation by $\rho \left(
\omega \right) =n_{\omega }$. We can consider the system of words $W=\left\{
w_{\omega }\mid \omega \in \Omega \right\} $. If $w_{\omega }\in F\left(
x_{1},\ldots ,x_{n_{\omega }}\right) =F_{n_{\omega }}\in \mathrm{Ob}\Theta
^{0}$ for every $\omega \in \Omega $, then for every $H\in \Theta $ we can
consider the new realization of our signature $\omega _{H}^{\ast }\left(
h_{1},\ldots ,h_{n_{\omega }}\right) =\left( w_{\omega }\right) _{H}^{\ast
}\left( h_{1},\ldots ,h_{n_{\omega }}\right) $ where $\omega \in \Omega $, $%
h_{1},\ldots ,h_{n_{\omega }}\in H$. The set $H$ with this new realization
of our signature we denote by $H_{W}^{\ast }$.

\begin{definition}
\label{applicable}\textit{We say that a system of words }$W=\left\{
w_{\omega }\mid \omega \in \Omega \right\} $ \textit{is \textbf{applicable}
if}
\end{definition}

1. \textit{for every }$\omega \in \Omega $\textit{\ the }$w_{\omega }\in
F_{n_{\omega }}$\textit{\ holds and}

2. \textit{there exists a system of bijections}%
\begin{equation*}
\left\{ s_{F\left( X\right) }:F\left( X\right) \rightarrow F\left( X\right)
\mid F\left( X\right) \in \mathrm{Ob}\Theta ^{0}\right\} 
\end{equation*}%
\textit{such that}%
\begin{equation*}
\left( s_{F\left( X\right) }\right) _{\mid X}=id_{X}
\end{equation*}%
\textit{and}%
\begin{equation*}
s_{F\left( X\right) }:F\left( X\right) \rightarrow \left( F\left( X\right)
\right) _{W}^{\ast }
\end{equation*}%
\textit{are isomorphisms.}

And the following proposition is central to the Method of verbal operations.

\begin{proposition}
There is a bijection between the group $\mathfrak{S}$ and the family of all
applicable systems of words. If $\Phi =\Phi _{W}$ is a strongly stable
automorphism of $\Theta ^{0}$, which corresponds to the applicable systems
of words $W$, the (\ref{autom_action}) holds for bijections which are
subjects of 2. from Definition \ref{applicable}. Vice versa, the applicable
systems of words $W_{\Phi }$, which corresponds to the strongly stable
automorphism $\Phi $ of $\Theta ^{0}$, is defined by%
\begin{equation*}
w_{\omega }\left( x_{1},\ldots ,x_{n_{\omega }}\right) =s_{F_{n_{\omega
}}}\left( \omega \left( x_{1},\ldots ,x_{n_{\omega }}\right) \right) \in
F_{n_{\omega }}
\end{equation*}%
for every $\omega \in \Omega $, where $s_{F_{\omega }}$ are subjects of 2.
from Definition \ref{str_st}.
\end{proposition}

Also we can conclude from \cite[Proposition 2.1]{FernandesTsur} the
following criterion\ for calculation the group $\mathfrak{S\cap Y}$.

\begin{criterion}
\label{inner_stable}We consider the strongly stable automorphism $\Phi $ of
the category $\Theta ^{0}$ and denote $W_{\Phi }=W$. This automorphism is
inner if and only if there exists $c(x_{1})\in F(x_{1})$ such that for every 
$F\left( X\right) \in \mathrm{Ob}\Theta ^{0}$ the mapping\linebreak $%
c_{F\left( X\right) }:F\left( X\right) \rightarrow \left( F\left( X\right)
\right) _{W}^{\ast }$, which fulfills $c_{F\left( X\right) }(f)=c(f)$ for
every $f\in F\left( X\right) $, is an isomorphism.
\end{criterion}

\begin{theorem}
\label{H_H_w_theorem}\cite[Theorem 5.1]{TsurAutomEqAlg}The automorphism $%
\Phi ^{-1}$, where $\Phi \in \mathfrak{S}$, provides the automorphic
equivalence of algebras $H$ and $H_{W}^{\ast }$, where $H\in \Theta $ and $%
W_{\Phi }=W$.
\end{theorem}

\begin{theorem}
\cite[Theorem 6.1]{TsurAutomEqAlg}Let algebras $H_{1}$ and $H_{2}$ belong to
the variety $\Theta $. They are automorphically equivalent in $\Theta $ if
and only if the algebra $H_{1}$ is geometrically equivalent to the algebra $%
(H_{2})_{W}^{\ast }$, where $W$ is a \textit{some }applicable\textit{\ }%
system of words.
\end{theorem}

\section{Examples, conjectures and directions for further research.}

The following example was considered in \cite[Section 5]{Tsurkov}.

\begin{example}
\normalfont$\Theta $ is a variety of $k$-linear algebras, where $k$ is a
field with characteristic $0$. This variety has the signature $\Omega
=\left\{ 0,+,\cdot ,\lambda \cdot \left( \lambda \in k\right) \right\} $,
where $0$ is a constant, addition and multiplication $+$ and $\cdot $ have
arity $2$ and $\lambda \cdot $ is an unary fixed scalar $\lambda $
multiplication operation for every $\lambda \in k$. In this variety we have
that $\mathfrak{A/Y\cong }k^{\ast }\mathfrak{\leftthreetimes }\mathrm{Aut}k$%
, where $k^{\ast }=k\setminus \left\{ 0\right\} $, $\mathrm{Aut}k$ is a
group of all automorphisms of the field $k$. Was proved that%
\begin{equation*}
W=\left\{ w_{0}=0,w_{+}\left( x_{1},x_{2}\right) =x_{1}+x_{2},\right. 
\end{equation*}%
\begin{equation*}
\left. w_{\cdot }\left( x_{1},x_{2}\right) =x_{2}\cdot x_{1},w_{\lambda
\cdot }\left( x_{1}\right) =\lambda x_{1}\left( \lambda \in k\right)
\right\} 
\end{equation*}%
is an applicable systems of words in this variety. $H$ is a two-generated
free algebra in the subvariety of $\Theta $ defined by identity $\left(
x_{1}x_{1}\right) x_{2}=0$. Algebras $H$ and $H_{W}^{\ast }$ are
automorphically equivalent by Theorem \ref{H_H_w_theorem}. But, it was
proved, that these algebras are not geometrically equivalent.
\end{example}

More examples of algebras, which are automorphically equivalent but are not
geometrically equivalent, were considered in \cite[Subsections 5.3 and 5.4]%
{TsurkovManySorted} and in \cite[Section 6]{TsurAutomEqVarLinAlg}.

Also there are many examples of varieties, which have a trivial group $%
\mathfrak{A/Y}$. In these varieties the automorphic equivalence coincides
with the geometric one. Some of these varieties were considered in \cite%
{TsurNilp} and \cite[Subsections 5.1 and 5.2]{TsurkovManySorted}.

A situation is also possible when in the some variety $\Theta $ the group $%
\mathfrak{A/Y}$ is not trivial, but the automorphic equivalence coincides
with the geometric one.

\begin{example}
\normalfont$\Theta $ is the variety of all vector spaces over fixed field $k$
with characteristic $0$. $\Omega =\left\{ 0,+,\lambda \cdot \left( \lambda
\in k\right) \right\} $. It is easy to prove that $\mathfrak{A/Y}\cong 
\mathrm{Aut}k$. In this variety%
\begin{equation*}
W_{\mathcal{\mu }}=\left\{ w_{0}=0,w_{+}\left( x_{1},x_{2}\right)
=x_{1}+x_{2},\right.
\end{equation*}%
\begin{equation*}
\left. w_{\lambda \cdot }\left( x_{1}\right) =\left( \mu \left( \lambda
\right) \right) x_{1}\left( \lambda \in k\right) \right\} ,
\end{equation*}%
where $\mu \in \mathrm{Aut}k$, is an applicable systems of words, which
corresponds to non-inner automorphism of $\Theta ^{0}$ when $\mu \neq id_{k}$%
. But, every two non-zero vector spaces are geometrically equivalent. It is
easy to prove this fact by \cite[Theorem 3]{PPT}, because all non-zero
vector spaces are direct products of some copies of one-dimensional vector
space.
\end{example}

We also have a

\begin{conjecture}
If $\Theta $ is the variety of all commutative associative linear algebras
over fixed field $k$ with characteristic $0$, then, despite the fact that $%
\mathfrak{A/Y}\cong \mathrm{Aut}k$, the automorphic equivalence coincides
with the geometric one in this variety.
\end{conjecture}

If $\Theta $ is the variety of all groups \cite{Aladova}, or is the variety
of all abelian groups, or is the variety of all nilpotent class $n$ groups
for every $n\in 
\mathbb{N}
$ \cite{TsurNilp}, then the group $\mathfrak{A/Y}$ is trivial. In 2007, B.
Plotkin raised the question: is there some subvariety of the variety of all
groups, in which the group $\mathfrak{A/Y}$ is not trivial? The variety of
all metaabelian, nilpotent class $4$ groups with exponent no more than $4$
was considered in \cite{FernandesTsur}. This variety has a signature $\Omega
=\left\{ 1,-1,\cdot ,\right\} $, where $1$ is a constant, $\cdot $ is a
multiplication and $-1$ is an unary operation which for every element of a
group gives us the inverse one. Was proved that the group $\mathfrak{A/Y}$
has two elements, because

\begin{equation*}
W=\left\{ w_{1}=1,w_{-1}\left( x_{1}\right) =x_{1}^{-1},\right.
\end{equation*}%
\begin{equation*}
\left. w_{\cdot }\left( x_{1},x_{2}\right) =x_{1}\cdot x_{2}\cdot \left(
x_{2},x_{1}\right) ^{2}\right\}
\end{equation*}%
is an applicable systems of words, which corresponds to non-inner
automorphism of $\Theta ^{0}$. But we couldn't give an example of two groups
of this variety, which are automorphically equivalent but are not
geometrically equivalent. We have a

\begin{conjecture}
In all subvarieties of the variety of all groups the automorphic equivalence
coincides with the geometric one.
\end{conjecture}

One of the interesting direction for further research is the study of gap
between automorphic equivalence and weak\textbf{\ }similarity. Now we don't
know any example of two algebras which are weak similar but are not
automorphically equivalent.

The other direction for further research is the consideration of many-sorted
algebras. We can consider in this algebras closers of equations systems of
some specific sorts and change the notion of algebraic closed sets. Will
this require changing our approach to automorphic equivalence?

\section{Acknowledgments}

I am thankful to Prof. G. Zhitomirski, whose definition of the category of $%
H $-closed congruences (see Subsection \ref{categories}) was very useful in
this research. I am also grateful to Dr. E. Aladova, Prof. E. Plotkin, and
Prof. G. Zhitomirski for numerous discussions that helped me in preparing
this article.


\begin{thebibliography}{99}
\bibitem{Aladova} E. Aladova, Method of Verbal Operations and Automorphisms
of Category of Free Algebras, \textit{Algebra and Logic}, \textbf{61:2}
(2022), pp. 87--103.

\bibitem{FernandesTsur} R. Barbosa Fernandes, A. Tsurkov, Automorphisms of
the category of finitely generated free groups of the some subvariety of the
variety of all groups. arXiv: 1909.05955V1 [math.GR].

\bibitem{Mal} A. I. Mal'cev, \textit{Algebraic Systems,} (Springer-Verlag,
Berlin, Heidelberg, New York, 1973).

\bibitem{PlotkinVarCat} B. Plotkin, Varieties of algebras and algebraic
varieties. Categories of algebraic varieties. \textit{Siberian Advanced
Mathematics, Allerton Press,} \textbf{7:2} (1997), pp. 64 -- 97.

\bibitem{PPT} B. Plotkin, E. Plotkin, A. Tsurkov, Geometrical equivalence of
groups. \textit{Communications in Algebra. }\textbf{27:8 }(1999), pp.
4015-4025.

\bibitem{PlotkinSame} B. Plotkin, Algebras with the same (algebraic)
geometry. \textit{Proceedings of the Steklov Institute of Mathematics.} 
\textbf{242} (2003), 17--207. DOI: 10.1134/S0081543812070048.

\bibitem{PlZhit} B. Plotkin, G. Zhitomirski, On automorphisms of categories
of free algebras of some varieties. J. Algebra \textbf{306}(2), 344--367
(2006).

\bibitem{TsurAutomEqAlg} A. Tsurkov, Automorphic equivalence of algebras, 
\textit{International Journal of Algebra and Computation.} \textbf{17:5/6},
(2007), pp. 1263 -- 1271.

\bibitem{TsurNilp} A. Tsurkov, Automorphisms of the category of the free
nilpotent groups of the fixed class of nilpotency. \textit{International
Journal of Algebra and Computation.} \textbf{17:5/6}, (2007), pp. 1273 --
1281.

\bibitem{Tsurkov} A. Tsurkov, Automorphic equivalence of linear algebras,%
\textit{\ Journal of Algebra and Its Applications}, \textbf{13:7}, (2014),
DOI: 10.1142/S0219498814500261.

\bibitem{TsurkovManySorted} A. Tsurkov. Automorphic equivalence of
many-sorted algebras. \textit{Applied Categorical Structures.} \textbf{24:3 }%
(2016), 209-240. DOI: 10.1007/s10485-015-9394-y.

\bibitem{TsurAutomEqVarLinAlg} A. Tsurkov, Automorphic equivalence in the
classical varieties of linear algebras, \textit{International Journal of
Algebra and Computation, }\textbf{27:8} (2017), pp. 973--999. DOI:
10.1142/S021819671750045X
\end{thebibliography}
\end{document}